\documentclass{amsart}
\usepackage[all]{xy}
\usepackage{graphicx}
\usepackage{amsmath}
\usepackage{amsfonts}
\usepackage{amssymb}

\usepackage[left=3.6cm, right=3.6cm]{geometry}

\newcommand{\qi}{\mathbf{i}}
\newcommand{\qj}{\mathbf{j}}
\newcommand{\qk}{\mathbf{k}}
\newcommand{\qzero}{\mathbf{0}}
\newcommand{\id}{\mathrm{id}}

\newtheorem{theorem}{Theorem}[section]

\newtheorem{corollary}[theorem]{Corollary}

\newtheorem{proposition}[theorem]{Proposition}

\theoremstyle{definition}
\newtheorem{definition}[theorem]{Definition}

\theoremstyle{remark}
\theoremstyle{remark}
\newtheorem{remark}[theorem]{Remark}

\numberwithin{equation}{section}

\begin{document}
\title[Topology of $3$-cosymplectic manifolds]{Topology of $3$-cosymplectic manifolds}

\author[B. Cappelletti Montano]{Beniamino Cappelletti Montano}
 \address{Dipartimento di Matematica e Informatica, Universit\`a degli Studi di
 Cagliari, Via Ospedale 72, 09124 Cagliari, Italy}
 \email{b.cappellettimontano@gmail.com}

\author[A. De Nicola]{Antonio De Nicola}
 \address{CMUC, Department of Mathematics, University of Coimbra, 3001-454 Coimbra, Portugal}
 \email{antondenicola@gmail.com}

\author[I. Yudin]{Ivan Yudin}
 \address{CMUC, Department of Mathematics, University of Coimbra, 3001-454 Coimbra, Portugal}
 \email{yudin@mat.uc.pt}



\thanks{The first author was partially supported by a research grant from Regione
Puglia. The second author is partially supported by the FCT grant
PTDC/MAT/099880/2008 and by MICIN (Spain) grant MTM2009-13383. The
last author is supported by the FCT Grant SFRH/BPD/31788/2006. All
the authors want to thank CMUC for hospitality and support.}

\begin{abstract}
We continue the program of Chinea, De Le\'on and Marrero who studied
the topology of cosymplectic manifolds. We study 3-cosymplectic
manifolds which are the closest odd-dimensional analogue of
hyper-K\"{a}hler structures. We  show that there is an action of the
Lie algebra $so(4,1)$ on the basic cohomology spaces of a compact
3-cosymplectic manifold with respect to the Reeb foliation. This
implies some topological obstructions to the existence of such
structures which are expressed by bounds on the Betti numbers. It is
known that every 3-cosymplectic manifold is  a local Riemannian
product of a hyper-K\"ahler factor and an abelian three dimensional
Lie group. Nevertheless, we present a nontrivial example of compact
3-cosymplectic manifold which is not the global product of a
hyper-K\"{a}hler manifold and a flat 3-torus.

\end{abstract}

\maketitle

\section{Introduction}
Cosymplectic geometry is considered to be the closest
odd-dimensional analogue of K\"{a}hler geometry (see e.g.
\cite[Section~6.5]{blairbook},
\cite[Section~14.5]{dragomirorneabook}). This becomes even more
evident when one passes to the setting of 3-structures. Indeed,
while both cosymplectic and Sasakian manifolds admit a transversal
K\"{a}hler structure, only 3-cosymplectic manifolds admit a
transversal hyper-K\"{a}hler structure (cf.
\cite{cappellettidenicola}).

In the fundamental paper \cite{chineamarrero}, Chinea, De Le\'on and
Marrero studied the topology of cosymplectic manifolds, refining the
previous results of Blair and Goldberg (\cite{blairgoldberg}). They
proved a monotonicity result for the Betti numbers of a compact
cosymplectic manifold $M^{2n+1}$ up to the middle dimension. Next,
the differences $b_{2p+1}-b_{2p}$ (with $0\leq p\leq n$) were shown
to be even integers (in particular, $b_{1}$ is odd). Moreover, they
found an example of a compact cosymplectic manifold which is not the
global product of a K\"{a}hler manifold and the circle. Later on,
other nontrivial examples were provided (cf.
\cite{marreropadron,FinoVezzoni}). More recently, Li (\cite{Li})
gave an alternative proof of the monotonicity property of the Betti
numbers of  cosymplectic manifolds (which he prefers to call
co-K\"{a}hler) by using topological techniques.

A 3-cosymplectic manifold (see e.g.
\cite[Section~13.1]{galickibook}) is a smooth manifold $M$ of
dimension $ 4n+3$ endowed with an almost contact metric 3-structure
such that each structure is cosymplectic. This class of Riemannian
manifolds is contained in the wider class of 3-quasi Sasakian
manifolds. Every 3-cosymplectic manifold is in particular
cosymplectic hence all the previously mentioned results still hold.
A natural problem is whether the quaternionic-like conditions which
relate the structure tensors of 3-cosymplectic manifolds can induce
additional rigidity to the underlying topological structure. The aim
of this paper is to give an answer to this question.

Every $3$-cosymplectic manifold $M$ admits the canonical Reeb
foliation $\mathcal{F}_3$ of dimension three. We denote by
$H_B^*\left( M \right)$ the basic cohomology with respect to this
foliation. The first result we prove in Section~\ref{decomposition}
can be restated in the form
\begin{equation}
    \label{kueneth}
    H_{dR}^*\left( M \right) \cong H^*_B\left( M \right) \otimes H_{dR}^*\left(
    \mathbb{T}^3 \right)
\end{equation}
    for any compact $3$-cosymplectic manifold $M$.
This shows that the Betti numbers of $M$ are completely determined by the basic
Betti numbers $b^h_p := \dim H^p_B\left( M \right)$, namely
\begin{equation}\label{formulamagica}
b_p=b_p^h+3b_{p-1}^h+3b_{p-2}^h+b_{p-3}^h.
\end{equation}

When $\mathcal{F}_3$ is a regular foliation we can identify
$H^*_B\left( M \right)$ with $H_{dR}^*\left( M/\mathcal{F}_3
\right)$. In this case $M/\mathcal{F}_3$ is a hyper-K\"ahler
manifold. There are known several results which give restrictions on
possible values of Betti numbers of compact  hyper-K\"ahler
manifolds. This suggests to look for the similar results about
$H^*_B\left( M \right)$. The results  on Betti numbers of compact
hyper-K\"ahler manifolds can be divided into two families. In one
family there are results that can be obtained from the existence of
the $so\left( 4,1 \right)$ action on the cohomology ring of a
hyper-K\"ahler manifold discovered by Verbitsky in \cite{verbitski}.
In the other family there are the equations derived from the
Riemann-Roch theorem by Salamon in~\cite{salamon}. There is no hope
at the moment to get an extension of the Salamon's result for
$H^*_B\left( M \right)$ for a case when $\mathcal{F}_3$ is
non-regular, as the theory of transversally hyper-K\"ahler
foliations is not developed enough.

In Section~\ref{action} we show the existence of an $so\left( 4,1 \right)$
action on $H^*_B\left( M \right)$. From representation theory of $so\left( 4,1
\right)$ it follows that the basic Betti numbers $b^h_{2p+1}$ are divisible by
four, and that
$$
b^h_{2p} \ge \binom{p+2}{2}, \ 0\le p\le n.
$$
We show these results in Section~\ref{quaternions} and
Section~\ref{inequalities} by more elementary arguments to make the article
accessible to a wider audience. As consequences, we will obtain that for a compact
$3$-cosymplectic manifold
$b_{2p} + b_{2p+1}$ are multiples of four and that
\[
b_{p} \geq \binom{p+2}{2}\qquad\quad \mbox{for }0\leq p\leq 2n+1.
\]

From the above considerations one can see that there are strong
obstructions to the existence of compact 3-cosymplectic manifolds.
On the other hand, every 3-cosymplectic manifold is  a local
Riemannian product of a hyper-K\"ahler factor and an abelian three
dimensional Lie group. Moreover, the formula \eqref{kueneth} could
suggest that every compact $3$-cosymplectic manifold is the  total
space of a toric bundle over a hyper-K\"ahler manifold. We disprove
this by an example in Section~\ref{example}. Namely, we construct a
compact $3$-cosymplectic seven dimensional manifold $M^7$ such that
$H^*_B\left( M^7 \right)$ cannot coincide with the cohomology ring
of a hyper-K\"ahler manifold. Note that in particular $M^7$ is not a
global product of  a hyper-K\"ahler manifold and $\mathbb{T}^3$,
which answers the open question about the existence of non-trivial
examples of such manifolds.

\section{Preliminaries}
An \emph{almost contact manifold} is an odd-dimensional manifold
$M$ which carries a field $\phi$ of endomorphisms of the tangent
spaces, a vector field $\xi$, called \emph{characteristic} or
\emph{Reeb vector field}, and a $1$-form $\eta$ satisfying
\begin{equation*}
\phi^2=-I+\eta\otimes\xi, \qquad \eta\left(\xi\right)=1,
\end{equation*}
where $I\colon TM\rightarrow TM$ is the identity mapping. From the
definition it follows that $\phi\xi=0$, $\eta\circ\phi=0$ and that
the $(1,1)$-tensor field $\phi$ has constant rank $2n$ (cf.
\cite{blairbook}). An almost contact  manifold
$\left(M,\phi,\xi,\eta\right)$ is said to be \emph{normal} when the
tensor field $N_{\phi}=\left[\phi,\phi\right]+2d\eta\otimes\xi$
vanishes identically, where $\left[\phi,\phi\right]$ is the
Nijenhuis torsion of $\phi$. It is known (see e.g. \cite[page
44]{blairbook}) that any almost contact manifold
$\left(M,\phi,\xi,\eta\right)$ admits a Riemannian metric $g$ such
that
\begin{equation}\label{compatibile}
g\left(\phi E,\phi F\right)=g\left(E,F\right)-\eta\left(E\right)\eta\left(F\right)
\end{equation}
holds for all $E,F\in\Gamma\left(TM\right)$. This metric $g$ is
called a \emph{compatible metric} and the manifold $M$ together with
the structure $\left(\phi,\xi,\eta,g\right)$ is called an
\emph{almost contact metric manifold}. As an immediate consequence
of \eqref{compatibile}, one has $\eta=g\left(\cdot,\xi\right)$ and
$g\left( \phi E, F \right)= - g\left( E,\phi F \right)$. Hence
$\Phi\left(E,F\right)=g\left(E,\phi F\right)$ defines a $2$-form,
which is called the \emph{fundamental $2$-form} of $M$.  Almost
contact metric manifolds such that both $\eta$ and $\Phi$ are closed
are called \emph{almost cosymplectic manifolds} and those for which
$d\eta=\Phi$ are called \emph{contact metric manifolds}. Finally, a
normal almost cosymplectic manifold is called a \emph{cosymplectic
manifold}, and a normal contact metric manifold is said to be a
\emph{Sasakian manifold}. In terms of the covariant derivative of
$\phi$, the cosymplectic and the Sasakian conditions can be
expressed respectively by
\begin{equation*}
\nabla\phi=0
\end{equation*}
and
\begin{equation*}
\left(\nabla_E\phi\right)F=g\left(E,F\right)\xi-\eta\left(F\right)E,
\end{equation*}
for all $E,F\in\Gamma\left(TM\right)$.

It should be noted that both
in Sasakian and in cosymplectic manifolds $\xi$ is a Killing vector
field. The Sasakian and the cosymplectic manifolds represent the
two extremal cases of the larger class of quasi-Sasakian manifolds
(cf. \cite{blairqs}).

An  \emph{almost contact $3$-structure} on  a
$\left(4n+3\right)$-dimensional smooth  manifold $M$
is given by three almost contact structures
$\left(\phi_1,\xi_1,\eta_1\right)$,
$\left(\phi_2,\xi_2,\eta_2\right)$,
$\left(\phi_3,\xi_3,\eta_3\right)$ satisfying the following
relations, for every $\alpha,\beta\in\left\{1,2,3\right\}$,
\begin{gather}
\phi_\alpha\phi_\beta-\eta_\beta\otimes\xi_\alpha = \sum_{\gamma=1}^{3}\epsilon_{\alpha\beta\gamma}\phi_\gamma - \delta_{\alpha\beta}I,\label{3-sasaki} \\
\phi_\alpha \xi_\beta = \sum_{\gamma=1}^{3}\epsilon_{\alpha\beta\gamma} \xi_\gamma, \quad \eta_\alpha\circ\phi_\beta =  \sum_{\gamma=1}^{3}\epsilon_{\alpha\beta\gamma} \eta_\gamma, \label{3-sasaki1}
\end{gather}
where $\epsilon_{\alpha\beta\gamma}$ is the totally antisymmetric
symbol. This notion was introduced  by Kuo (\cite{kuo}) and,
independently, by Udriste (\cite{udriste}). In \cite{kuo} Kuo proved
that given an almost contact $3$-structure
$\left(\phi_\alpha,\xi_\alpha,\eta_\alpha\right)$, $\alpha\in
\{1,2,3\}$,  there exists a Riemannian metric $g$ compatible with
each of the structures and hence we can speak of \emph{almost
contact metric $3$-structure}. It is well known that in any almost
$3$-contact  manifold the Reeb vector fields $\xi_1,\xi_2,\xi_3$ are
orthonormal with respect to any compatible metric $g$ and that the
structural group of the tangent bundle is reducible to
$Sp\left(n\right)\times \{I_3\}$. Moreover, the tangent bundle of
any almost 3-contact metric manifold splits up as the orthogonal sum
$TM={\mathcal H}\oplus{\mathcal V}$, where the $4n$-dimensional
subbundle
${\mathcal{H}}=\bigcap_{\alpha=1}^{3}\ker\left(\eta_\alpha\right)$
is called the \emph{horizontal distribution} and ${\mathcal
V}=\left\langle\xi_1,\xi_2,\xi_3\right\rangle$ is  called the
\emph{vertical} (or \emph{Reeb}) \emph{distribution}. An almost
$3$-contact manifold $M$  is  said  to  be \emph{normal}  if each
almost contact  structure
$\left(\phi_\alpha,\xi_\alpha,\eta_\alpha\right)$ is normal.

Let $\left(\phi_\alpha,\xi_\alpha,\eta_\alpha,g\right)$ be an almost
contact metric $3$-structure. When each structure is Sasakian $M$ is
called a \emph{$3$-Sasakian manifold}.

By  an   \emph{almost $3$-cosymplectic manifold}   we  mean an
almost  $3$-contact  metric  manifold $M$ such that each almost
contact metric structure
$\left(\phi_\alpha,\xi_\alpha,\eta_\alpha,g\right)$ is almost
cosymplectic. The almost cosymplectic $3$-structure
$\left(\phi_\alpha,\xi_\alpha,\eta_\alpha,g\right)$ is called
\emph{cosymplectic} if it is normal. In this case $M$ is said to be
a \emph{$3$-cosymplectic manifold}. However it has been proved
recently in \cite[Theorem 4.13]{pastore} that these two notions are
the same, i.e. every almost $3$-cosymplectic manifold is
$3$-cosymplectic.

Just as in the case of a single structure, the 3-Sasakian and the
3-cosymplectic manifolds represents the two extremal cases of the
larger class of 3-quasi-Sasakian manifolds (cf.
\cite{cappellettidenicoladileo2}).

In any $3$-cosymplectic manifold the forms $\eta_\alpha$ and
$\Phi_\alpha$ are harmonic (\cite[Lemma~3]{goldbergyano}). Moreover, we have
that $\xi_\alpha$, $\eta_\alpha$, $\phi_\alpha$ and $\Phi_\alpha$
are $\nabla$-parallel. In particular
\begin{equation}\label{commutatore}
    \left[\xi_\alpha,\xi_\beta\right] = \nabla_{\xi_\alpha}\xi_\beta-\nabla_{\xi_\beta}\xi_\alpha = 0
\end{equation}
for all $\alpha,\beta\in\left\{1,2,3\right\}$, so that $\mathcal V$ defines a
$3$-dimensional foliation ${\mathcal F}_3$ of $M^{4n+3}$. Since each Reeb vector field is Killing and is parallel, such a foliation turns out to be Riemannian with totally geodesic leaves.

Recall that a foliation $\mathcal{F}$ is \emph{regular} (in the
sense of Palais~\cite{palais}) if each point $p\in M$ has a foliated
coordinate chart $\left( U,p \right)$ such that each leaf of
$\mathcal{F}$ passes through $U$ at most once.

\begin{theorem}\emph{(\cite[Corollary~3.10]{cappellettidenicola})}\label{projectable}
Let $\left(M^{4n+3},\phi_\alpha,\xi_\alpha,\eta_\alpha,g\right)$
be a $3$-cosymplectic manifold.
If the foliation $\mathcal{F}_3 $ is regular, then the space
of leaves $M^{4n+3}/{\mathcal{F}_3 }$ is a hyper-K\"{a}hler manifold of dimension $4n$.
Consequently, every $3$-cosymplectic manifold is Ricci-flat.
\end{theorem}

\begin{remark}
If we drop the assumption of regularity in Theorem \ref{projectable} and we assume instead that the vertical foliation has compact leaves, then the space of leaves is a hyper-K\"{a}hler orbifold, i.e. a second countable Hausdorff space locally modeled on
finite quotients of $\mathbb{R}^{m}$. We refer to \cite{molino} for the formal definition and properties of orbifolds and to \cite{satake}
for the generalization of geometric objects to the orbifold category.
\end{remark}

Concerning the horizontal subbundle, note that --- unlike the case
of $3$-Sasakian geometry --- in any 3-cosymplectic manifold
$\mathcal H$ is integrable. Indeed, for all
$X,Y\in\Gamma\left(\mathcal H\right)$,
$\eta_\alpha\left(\left[X,Y\right]\right)=-2d\eta_\alpha\left(X,Y\right)=0$
since $d\eta_\alpha=0$.

\section{Decomposition of the cohomology of 3-cosymplectic manifolds}
\label{decomposition}
Unless otherwise stated, in the remaining of the paper we will assume that
all manifolds are compact.
In this section we investigate some algebraic properties of the de
Rham cohomology $H^*_{dR}\left( M \right)$ of a $3$-cosymplectic
manifold $M^{4n+3}$. By the Hodge-de Rham theory the vector space
$H^k_{dR}\left( M\right)$ can be identified with the vector space
$\Omega^k_{H}\left( M \right)$ of harmonic $k$-forms on $M$.

For each $\alpha\in \left\{ 1,2,3 \right\}$ we define linear operators $\lambda_\alpha$
and $l_\alpha$ by
\begin{align*}
    l_\alpha \colon \Omega^k\left( M \right)& \to \Omega^{k+1}\left(
    M \right) & \lambda_\alpha \colon \Omega^{k+1}\left(M  \right) &\to
    \Omega^k\left( M \right)\\
    \omega &\mapsto  \eta_\alpha\wedge \omega & \omega&\mapsto
    i_{\xi_\alpha} \omega.
\end{align*}
We denote by $\left\{ A,B \right\}$ the anticommutator $AB+BA$ of
two linear operators $A$ and $B$. From $\eta_\alpha\left( \xi_\beta
\right) = \delta_{\alpha\beta}$ it follows that
\begin{align}
    \label{eq:lambdal}
    \left\{ \lambda_\alpha,l_\beta \right\} = \delta_{\alpha\beta}.
\end{align}
Moreover
\begin{align}
    \label{eq:anticom}
    \left\{ \lambda_\alpha,\lambda_\beta \right\} = \left\{ l_\alpha,l_\beta
    \right\} = 0.
\end{align}
Define $e_\alpha= l_\alpha \lambda_\alpha$. Then it follows from
\eqref{eq:lambdal} that $e_\alpha$ are idempotents. In fact
$$
e_\alpha e_\alpha = l_\alpha\lambda_\alpha l_\alpha\lambda_\alpha =-
l_\alpha l_\alpha \lambda_\alpha \lambda_\alpha + l_\alpha
\lambda_\alpha = e_\alpha.
$$
Moreover from \eqref{eq:lambdal} and \eqref{eq:anticom} it follows
that $\left[ e_\alpha,e_\beta \right] = 0$, for $\alpha\not =
\beta$. Thus $\left\{ e_1,e_2,e_3 \right\}$ are pairwise commuting
idempotents.

By \cite[Proposition~1]{chineamarrero} all operators $l_\alpha$,
$\lambda_\alpha$, and thus $e_\alpha$, preserve harmonic forms. Now
we fix $k\in \left\{ 0,\dots,4n+3 \right\}$ and consider the
restrictions of the operators $e_\alpha$ to $\Omega^k_H\left( M
\right)$, $\alpha\in\left\{ 1,2,3 \right\}$. Note that
$\Omega^k_H\left( M \right)$ is a finite dimensional vector space
over $\mathbb{R}$. As $e_\alpha$ is idempotent, its minimal
polynomial $m_\alpha\left( x \right)$ is a divisor of $x\left( x-1
\right)$. Therefore the only possible eigenvalues of $e_\alpha$ are
$0$ and $1$. Moreover, since $m_\alpha\left( x \right)$ does not
have multiple roots, the operator $e_\alpha$ is diagonalizable with
$0$ and $1$ on the diagonal.  As the operators $\left\{ e_1,e_2,e_3
\right\}$ commute with each other, by Bourbaki
\cite[Proposition~VII.13]{bourbakiAlgebraII} they can be
simultaneously diagonalized. Define for all triples
$\varepsilon_1$,~$\varepsilon_2$,~$\varepsilon_3\in \left\{ 0,1
\right\}$
$$
\Omega^k_{H,\varepsilon_1\varepsilon_2\varepsilon_3}\left( M \right)
= \left\{\, \omega\in \Omega^k_H\left( M \right) \,\middle|\,
e_\alpha \omega = \varepsilon_\alpha \omega,\  \alpha=1,2,
3\right\}.
$$
Since $e_1$, $e_2$, $e_3$ can be simultaneously diagonalized on
$\Omega^k_H\left( M \right)$ we get that
\begin{align}
    \label{eq:decomposition}
\Omega^k_H\left( M \right)  =
\bigoplus_{\varepsilon_1,\varepsilon_2,\varepsilon_3\in \left\{ 0,1
\right\}} \Omega^k_{H,\varepsilon_1\varepsilon_2\varepsilon_3}\left(
M \right).
\end{align}

Now let $\omega\in \Omega^k_{H, 0\varepsilon_2\varepsilon_3}\left( M
\right)$. Then $l_1 \omega\in
\Omega^{k+1}_{H,1\varepsilon_2\varepsilon_3}$. In fact
\begin{align*}
    e_1 l_1 \omega &  = l_1 \lambda_1 l_1 \omega = -
    \lambda_1 l_1 l_1 \omega + l_1 \omega = l_1\omega\\
    e_\alpha l_1 \omega & = l_1 e_\alpha \omega = \varepsilon_\alpha l_1
    \omega,\ \ \ \    \alpha = 2,3.
\end{align*}
Similarly if $\omega \in
\Omega^{k+1}_{H,1\varepsilon_2\varepsilon_3}\left( M \right) $, then
$\lambda_1 \omega \in
\Omega^{k}_{H,0\varepsilon_2\varepsilon_3}\left( M \right)$.
Therefore, we get maps of vector spaces
\begin{align*}
    l_1^{\varepsilon_2\varepsilon_3} \colon
    \Omega^k_{H,0\varepsilon_2\varepsilon_3}\left( M \right) \to
\Omega^{k+1}_{H,1\varepsilon_2\varepsilon_3}\left( M \right), &&
\lambda_1^{\varepsilon_2\varepsilon_3} \colon
\Omega^{k+1}_{H,1\varepsilon_2\varepsilon_3}\left( M \right) \to
\Omega^k_{H,0\varepsilon_2\varepsilon_3}\left( M \right).
\end{align*}
Now
$l_1^{\varepsilon_2\varepsilon_3}\lambda_1^{\varepsilon_2\varepsilon_3}
$ is the restriction of $e_1$ to
$\Omega^{k+1}_{H,1\varepsilon_2\varepsilon_3}\left( M \right)$ and
thus
$l_1^{\varepsilon_2\varepsilon_3}\lambda_1^{\varepsilon_2\varepsilon_3}=\id$.
Analogously the composition $\lambda_1^{\varepsilon_2\varepsilon_3}
l_1^{\varepsilon_2\varepsilon_3} $ is the restriction of
$$
\lambda_1 l_1 = \id - l_1\lambda_1  = \id - e_1
$$
to $\Omega^k_{H,0\varepsilon_2\varepsilon_3}\left( M \right) $ and
thus $\lambda_1^{\varepsilon_2\varepsilon_3}
l_1^{\varepsilon_2\varepsilon_3} = \id$. Thus
$\lambda_1^{\varepsilon_2\varepsilon_3}$ and
$l_1^{\varepsilon_2\varepsilon_3} $ are inverse isomorphisms between
the vector spaces $\Omega^k_{H,0\varepsilon_2\varepsilon_3}\left( M
\right) $ and $\Omega^{k+1}_{H,1\varepsilon_2\varepsilon_3}\left( M
\right)$. Replacing $1$ with $2$, $3$, and putting all together we
get for every $0\le k\le 4n$ the cube
$$
\xymatrix{ & \Omega^{k+1}_{H,100}\left( M \right) \ar[rr]^{l_2}
\ar[dd]^(.3){l_3}|!{[dl];[dr]}\hole &&
\Omega^{k+2}_{H,110} \left( M \right) \ar[dd]_{l_3}\\
\Omega^k_{H,000}\left( M \right)
\ar[rr]^(.7){l_2}\ar[dd]_{l_3}\ar[ru]^{l_1} &&
\Omega^{k+1}_{H,010}\left( M
\right)\ar[dd]^(.7){l_3} \ar[ru]^{l_1}\\
& \Omega^{k+2}_{H,101}\left( M \right)
{\ar[rr]^(.4){l_2}|!{[dr];[ur]}\hole}   &&
\Omega^{k+3}_{H,111}\left( M \right)\\
\Omega^{k+1}_{H,001}\left( M \right) \ar[rr]^{l_2}\ar[ru]^{l_1} &&
\Omega^{k+2}_{H,011}\left( M \right) \ar[ru]_{l_1} }
$$
whose faces are anti-commutative and edge arrows are isomorphisms of
vector spaces. Therefore the whole information about the cohomology
groups of $M$ is contained in the vector spaces $\Omega^k_{H,000}(M)$,
$0\le k\le 4n$.

Denote by $b_k^h$ the dimension of $\Omega^k_{H,000}(M)$. Then
\begin{align*}
    \dim \Omega^k_{H,100}& = \dim \Omega^k_{H,010} = \dim \Omega^k_{H,001} =
    \dim \Omega^{k-1}_{H,000} = b_{k-1}^h &  k& \ge 1\\
    \dim \Omega^k_{H,110} & = \dim \Omega^k_{H,101} = \dim
    \Omega^k_{H,011} = \dim \Omega^{k-2}_{H,000} = b_{k-2}^h & k& \ge 2\\
    \dim\Omega^k_{H,111} & = \dim \Omega^{k-3}_{H,000} = b_{k-3}^h & k&\ge 3.
\end{align*}
Therefore, from the decomposition \eqref{eq:decomposition} we get
    \begin{align}
        \label{eq:betti}
        \nonumber
        b_0 & = b^h_0\\
        b_1 & = b^h_1 + 3 b^h_0 \\
    \nonumber
        b_2 & = b^h_2 + 3 b^h_1 + 3 b^h_0\\
    \nonumber
        b_k & = b^h_k + 3 b^h_{k-1} + 3b^h_{k-2} + b^h_{k-3} && 3\le k\le 4n+3.
\end{align}

Now we will identify the vector spaces $\Omega^k_{H,000}\left( M
\right)$ with the basic cohomology of the Reeb foliation generated
by $\xi_\alpha$, $\alpha\in \left\{ 1,2,3 \right\}$,  on $M$.  In
our case the spaces of basic forms are given by
$$
\Omega^k_B\left( M \right) := \left\{\, \omega\in \Omega^k\left( M
\right) \,\middle|\, i_{\xi_\alpha} \omega = 0,\ i_{\xi_{\alpha}}d\omega =0
\mbox{ for each $\alpha=1,2,3$}
\right\}.
$$
The basic differential $d_B$ is the restriction of the exterior
derivative $d$ to $\Omega^*_B\left( M \right)$. The basic cohomology
spaces are defined as cohomology spaces of the complex $\left(
\Omega^*_B\left( M \right), d_B \right)$. In our case the mean
curvature of the Reeb foliation $\mathcal{F}_3$ is zero since the foliation is totally geodesic, therefore
we can use the transversal de Rham-Hodge theory developed
in~\cite{kamber}. By this theory, the basic cohomology spaces can be
identified with the kernel of the basic Laplacian
$$
\triangle_B := d_B \delta_B + \delta_B d_B,
$$
where $\delta_B $ is  the codifferential $\delta$ followed by the
orthogonal projection of $\Omega^*\left( M \right)$ onto $\Omega^*_B\left( M
\right)$.
We denote by $\Omega^*_{BH}\left( M \right)$ the kernel of
$\triangle_B$.
\begin{proposition}
    Let $M^{4n+3}$ be a compact $3$-cosymplectic manifold. Then
    $\Omega^*_{BH}\left( M \right) = \Omega^*_{H,000}\left( M \right)$.
    In particular, the numbers $b^h_k$ coincide with the basic Betti numbers
    of the Reeb foliation on $M$.
\end{proposition}

\begin{proof}
    First we show that $\Omega^*_{H,000}\left( M \right)\subset \Omega^*_B\left( M
    \right)$. Let $\omega\in \Omega^k_{H,000}\left( M \right)$.
    Then
    $$
    i_{\xi_\alpha}\omega = \lambda_\alpha\omega = \left( \lambda_\alpha
    -l_\alpha\lambda_\alpha^2 \right)\omega = \lambda_\alpha l_\alpha
    \lambda_\alpha \omega = \lambda_\alpha e_\alpha \omega=0,\ \ \   \alpha\in
    \left\{ 1,2,3 \right\}.
    $$
    Moreover, $d\omega =0$ therefore $\omega\in \Omega^k_B\left( M
    \right)$. Thus we have to show that $\Omega^*_{H,000}\left( M
    \right)$ is the kernel of $\triangle_B$. We know that
    $\Omega^*_{H,000}\left( M \right)$ is the kernel of $\triangle$. Thus it
    is enough to
     show that $\triangle_B = \triangle$ on $\Omega^*_B\left( M
    \right)$.
    From the definitions of $\triangle$ and $\triangle_B$ we see that it is
    enough to check that $\delta = \delta_B$ on $\Omega^*_B\left( M
    \right)$.
Recall, that $\delta_B$ is the restriction of $\delta$ to $\Omega^*_B\left( M
\right)$ followed by the orthogonal projection from $\Omega^*\left( M
\right)$ to $\Omega^*_B\left( M \right)$. Therefore, the map $\delta_B$
coincides with the restriction of $\delta$ to $\Omega^*_B\left( M \right)$ if
and only if $\delta\left( \Omega^*_B\left( M \right) \right)\subset
\Omega^*_B\left( M \right)$.

    Let $\omega\in \Omega^k_B\left( M \right)$. The operators
    $l_\alpha$ and $d$ anticommute in our case, since $l_\alpha$
    is the  wedge product with a closed $1$-form.
    As shown in~\cite[pages 97-98]{goldberg},  on a  Riemannian
    manifold the usual operator of  interior product $i_X$, where $X$ is  a vector
    field, can be defined as the Hodge dual of the
    operator $g\left( X,- \right)\wedge -$.
    Thus $\lambda_\alpha=i_{\xi_\alpha}$ is the Hodge dual of $l_\alpha=\eta_\alpha\wedge
    -$. Since $\delta$ is the Hodge dual of $d$ we get that
    $\delta$ and $\lambda_\alpha$ anticommute, which implies
\begin{align}\label{eq:ixidelta}
    i_{\xi_{\alpha}} \delta \omega & = -\delta i_{\xi_{\alpha}} \omega = 0.
\end{align}
Now we use the fact that the anticommutator of $i_{\xi_\alpha}$ and
$d$ is the Lie derivative $\mathcal{L}_{\xi_{\alpha}}$. In the last
paragraph of page 109 of~\cite{goldberg}, it is shown that for a
Killing vector field $X$
$$
\mathcal{L}_X + \{\delta, g\left( X,- \right)\wedge-\} = 0.
$$
Since $\xi_{\alpha}$ is a
Killing vector field, we get
$$
\mathcal{L}_{\xi_{\alpha}} + \{\delta, l_\alpha\} = 0.
$$
Therefore, $\delta$ and $\mathcal{L}_{\xi_{\alpha}}$ commute
$$
\left[ \delta, \mathcal{L_{\xi_\alpha}} \right] =- \left[ \delta, \left\{ \delta,
l_\alpha
\right\} \right] = -\delta^2 l_\alpha - \delta l_\alpha \delta + \delta l_\alpha
\delta + l_\alpha \delta^2 = 0
$$
and by \eqref{eq:ixidelta} we get
\begin{align*}
    i_{\xi_{\alpha}} d \delta \omega & = \mathcal{L}_{\xi_{\alpha}} \delta
    \omega - d i_{\xi_\alpha } \delta \omega
      = \delta \mathcal{L}_{\xi_{\alpha}} \omega + 0
     = \delta \left( di_{\xi_{\alpha}} + i_{\xi_{\alpha}} d \right) \omega
      = 0.
\end{align*}
In the last step we use that $\omega$ is basic.
Thus if $\omega\in \Omega^*_B\left( M \right)$ then $\delta \omega \in
\Omega^*_B\left( M \right)$. This concludes the proof.

\end{proof}

    \section{Action of $so\left( 4,1 \right)$ on the cohomology of 3-cosymplectic manifolds}
    \label{action}
In this section we will show that $\Omega^k_{H,000}\left( M \right)$
admits an action of the Lie algebra $so\left( 4,1 \right)$. This
result is the odd-dimensional analogue of the one obtained by Verbitsky
in~\cite{verbitski} about the action of $so\left( 4,1 \right)$ on
the cohomology groups of a hyper-K\"ahler manifold $M^{4n}$. In fact, intuitively
the space
$\bigoplus_{k=0}^{4n}\Omega^k_{H,000}\left( M \right)$ can be
thought of as a cohomology ring of the hyper-K\"ahler orbifold obtained
from $M^{4n+3}$ by taking the quotient under the action of the three
Reeb vector fields.

For every cyclic permutation $\left( \alpha,\beta,\gamma \right)$ of
$\left( 1,2,3 \right)$ we denote by $\Xi_{\alpha}$ the $2$-form
\begin{eqnarray}
    \label{eq:xi}
\Xi_\alpha := \frac 12 \left(\Phi_\alpha + 2 \eta_\beta\wedge
\eta_\gamma\right).
\end{eqnarray}
    Define the operators $L_\alpha\colon \Omega^k\left( M \right)\to
\Omega^{k+2}\left( M \right)$ and $\Lambda_\alpha  \colon
\Omega^{k+2}\left( M \right)\to \Omega^k \left( M \right)$ by $L_\alpha \omega =  \Xi_\alpha
\wedge \omega$ and $\Lambda_\alpha := * L_\alpha *$.

We will give now a local description of these operators.
Let $$\left\{ X_1,\phi_1 X_1,\phi_2 X_1, \phi_3
X_1, \dots, X_n, \phi_1 X_n, \phi_2 X_n, \phi_3 X_n,
\xi_1,\xi_2,\xi_3 \right\}$$ be an orthonormal basis of vector
fields in some open subset $U$  of $M$. Denote by $\zeta_s$ the
$1$-form dual to $X_s$, that is $\zeta_s = g\left( X_s, - \right)$.
Then
\begin{align}
    \label{eq:zeta}
    i_{\phi_\alpha X_s} \left( \phi_\alpha^* \zeta_t \right) = g\left( X_s,
    \phi_\alpha \left( \phi_\alpha X_t \right)
    \right) = g\left( X_s, \phi_\alpha^2 X_t \right)= - \delta_{st},\qquad 1\leq s,t\leq n.
\end{align}
Therefore the set
\begin{align}
    \label{eq:basis}
    \left\{ \zeta_1,\phi_1^* \zeta_1,\phi_2^* \zeta_1,\phi_3^*
    \zeta_1,\dots, \zeta_n , \phi_1^* \zeta_n, \phi_2^* \zeta_n, \phi_3^*
    \zeta_n, \eta_1, \eta_2,\eta_3 \right\}
\end{align}
is a basis of $1$-forms on $U$.

\begin{proposition}
    \label{Phi}
    Let $(\alpha,\beta,\gamma)$ be a cyclic permutation of $(1,2,3)$. Then

\begin{equation}
    \label{eq:Phi}
    \Phi_\alpha = 2  \sum_{s=1}^n \left( \zeta_s \wedge \phi_\alpha^*
    \zeta_s - \phi_\beta^* \zeta_s \wedge \phi_\gamma^* \zeta_s\right) -
    2\eta_\beta\wedge \eta_\gamma
\end{equation}
and therefore
\begin{equation}
    \label{eq:Xi}
    \Xi_\alpha =   \sum_{s=1}^n \left( \zeta_s \wedge \phi_\alpha^*
    \zeta_s - \phi_\beta^* \zeta_s \wedge \phi_\gamma^* \zeta_s\right).
\end{equation}

\end{proposition}
\begin{proof}
Let us denote by $\left\langle \ ,\  \right\rangle$ the natural
pairing between $k$-forms and $k$-vector fields. By definition of
$\Phi_\alpha$ we have
\begin{align*}
\left\langle  \Phi_\alpha, X_s\wedge \phi_\alpha X_s \right\rangle &
= g \left( X_s,
    \phi_\alpha^2 X_s
    \right) = -1\\
\left\langle  \Phi_\alpha,  \phi_\beta X_s\wedge \phi_\gamma X_s
\right\rangle   & = g\left(
    \phi_\beta X_s, \phi_\alpha \phi_\gamma X_s \right) = g\left( \phi_\beta
    X_s, - \phi_\beta X_s
    \right) = -1\\
\left\langle  \Phi_\alpha , \eta_\beta\wedge \eta_\gamma
\right\rangle & = g\left( \eta_\beta,
    \phi_\alpha \eta_\gamma
    \right) = g \left( \eta_\beta , -\eta_\beta  \right) = -1,
\end{align*}
and $\left\langle \Phi_\alpha,V \right\rangle =0$ for any other
element $V$ of the basis of the space of bivector fields on $U$. On
the other hand,
\begin{align*}
    \left\langle \zeta_s\wedge \phi_\alpha^* \zeta_s, X_s\wedge \phi_\alpha
    X_s\right\rangle & = \frac{1}{2} \zeta_s\left( X_s \right)
    \phi_\alpha^* \zeta_s
    \left( \phi_\alpha X_s \right) = -\frac 12\\
    \left\langle \phi_\beta^*\zeta_s\wedge \phi_\gamma^* \zeta_s, \phi_\beta
    X_s\wedge \phi_\gamma X_s \right\rangle  &= \frac 12  \phi_\beta^*
    \zeta_s \left( \phi_\beta X_s \right) \phi_\gamma^*\zeta_s \left(
    \phi_\gamma X_s \right) = \frac 12\\
    \left\langle \eta_\beta\wedge \eta_\gamma, \xi_\beta \wedge \xi_\gamma
    \right\rangle &= \frac 12 \eta_\beta\left( \xi_\beta \right)
    \eta_\gamma \left( \xi_\gamma \right) = \frac 12.
\end{align*}
\end{proof}
Note that for any $k$-form $\omega$ on $M$, any vector field $Y$ of
unit norm, and $\rho$ the dual $1$-form such that $\rho\left( Y
\right)=1$, we have
\begin{align}
    \label{eq:star}
*\left(\rho\wedge * \omega\right) = \left( -1 \right)^{\left( 4n+3-k
\right)\left( k-1
\right)} i_{Y} \omega.
\end{align}
    From \eqref{eq:zeta},
    \eqref{eq:Xi}, \eqref{eq:star},  and the fact that $*^2 = \id$ for odd dimensional manifolds, it
is easy to obtain the formula
\begin{equation}
    \label{eq:Lambda}
    \Lambda_\alpha = \sum_{s=1}^n \left( i_{X_s} i_{\phi_\alpha X_s} +
    i_{\phi_\beta X_s}i_{\phi_\gamma X_s} \right).
\end{equation}

\begin{remark}
    \label{horizontal}
    From \cite[Lemma~2.3]{blairgoldberg} it follows that the
operators $\omega\mapsto \Phi_\alpha \wedge \omega$ preserve
harmonic forms. Since the operator $\omega \mapsto \eta_\beta \wedge
\eta_\gamma \wedge \omega$ is equal to $l_\beta l_\gamma$, it also
preserves harmonicity. Then, by definition of the operators
$L_\alpha$, they preserve harmonicity as well.
 Since the Hodge star $*$ preserves
harmonic forms we get that also $\Lambda_\alpha$ preserves them.

 Now we verify that $L_\alpha\left( \Omega^*_{H,000}\left( M
\right)
\right)\subset \Omega^*_{H,000}\left( M \right)$ and
$\Lambda_\alpha\left( \Omega^*_{H,000}\left( M \right)
\right)\subset \Omega^*_{H,000}\left( M \right)$. For this it is enough to show
that $L_\alpha$ and $\Lambda_\alpha$ commute with $e_\mu$ for any pair
$1\le \alpha,\mu\le 3$. Since $\Lambda_\alpha$ is the Hodge dual of $L_\alpha$
and $\id - e_\mu$ is the Hodge dual of $e_\mu$, it is enough to check that
$L_\alpha$ commute with $e_\mu$.
We know that $e_\mu = l_\mu \lambda_\mu$. Since $l_\mu$ is the wedge
product with a $1$-form and $L_\alpha$ is the wedge product with a $2$-form,
they commute. Now, let $\omega\in \Omega^k\left( M \right)$, then
\begin{align*}
    \lambda_\mu L_\alpha \omega = i_{\xi_\mu}\left( \Xi_\alpha \wedge \omega
    \right) = \left( i_{\xi_\mu} \Xi_\alpha \right) \wedge \omega + \Xi_\alpha
    \wedge \left( i_{\xi_\mu} \omega \right) = \left( i_{\xi_\mu} \Xi_\alpha \right)
    \wedge \omega + L_\alpha \lambda_\mu \omega
\end{align*}
and by
\eqref{eq:Xi}
\begin{align*}
    i_{\xi_\mu} \Xi_\alpha = \sum_{s=1}^n  \left( i_{\xi_\mu} \zeta_s \wedge
    \phi^*_\alpha \zeta_s - \zeta_s \wedge i_{\xi_\mu} \phi^*_\alpha \zeta_s - i_{\xi_\mu}
    \phi^*_\beta \zeta_s \wedge \phi^*_\gamma \zeta_s + \phi^*_\beta\zeta_s
    \wedge i_{\xi_\mu} \phi^*_\gamma \zeta_s\right) = 0.
\end{align*}
As
consequence, we can restrict the operators $L_\alpha$ and
$\Lambda_\alpha$ to $\Omega^*_{H,000}\left( M \right)$. From now on,
we will consider $L_\alpha$ and $\Lambda_\alpha$ as endomorphisms of
$\Omega^*_{H,000}\left( M \right)$.
\end{remark}
    Define the operator $H\colon \Omega^k_{H,000}\left( M \right)\to
\Omega^k_{H,000}\left( M \right)$ by $H\omega = \left(2 n-k \right)
\omega$.
\begin{proposition}
    \label{llambdah}
    We have $\left[ L_\alpha, \Lambda_\alpha \right] =- H$ on
    $\Omega^*_{H,000}\left( M \right)$.
\end{proposition}
\begin{proof}
    Every element of $\Omega^k_{H,000}$ can be locally written   as a linear combination of
    wedges of elements in
    \begin{equation}
        \label{eq:basisshort}
            \left\{ \zeta_1,\phi_1^* \zeta_1,\phi_2^*
            \zeta_1,\phi_3^*
            \zeta_1,\dots, \zeta_n , \phi_1^* \zeta_n, \phi_2^*
            \zeta_n, \phi_3^*
            \zeta_n\right\}.
        \end{equation}
Note that for any $1\le s,t\le n$
and any cyclic permutation $\left(\alpha, \beta, \gamma\right)$ of $\left(1, 2,
3\right)$
\begin{align*}
    \left[ \zeta_s\wedge \phi_\alpha^* \zeta_s\wedge - ,
i_{\phi_\beta
X_t}i_{\phi_\gamma X_t}
     \right] & = 0 &
\left[  \phi_\beta^* \zeta_s
        \wedge \phi_\gamma^* \zeta_s \wedge -,
    i_{X_t}i_{\phi_\alpha X_t}
    \right]= 0
 \end{align*}
 and for $s\not= t$
 \begin{align*}
     \left[ \zeta_s\wedge \phi_\alpha^* \zeta_s\wedge - ,
    i_{X_t}i_{\phi_\alpha X_t}
     \right] & = 0 &
\left[  \phi_\beta^* \zeta_s
        \wedge \phi_\gamma^* \zeta_s \wedge -,
i_{\phi_\beta
X_t}i_{\phi_\gamma X_t}
    \right]= 0.
 \end{align*}
 Therefore by \eqref{eq:Xi} and \eqref{eq:Lambda}, we get
    \begin{align}
        \label{eq:commutator}
        \left[ L_\alpha, \Lambda_\alpha \right]&= \sum_{s=1}^n\left( \left[
        \zeta_s\wedge \phi^*_\alpha \zeta_s\wedge - ,
        i_{X_s}i_{\phi_\alpha X_s} \right] - \left[ \phi_\beta^* \zeta_s
        \wedge \phi_\gamma^* \zeta_s \wedge -, i_{\phi_\beta
        X_s}i_{\phi_\gamma X_s}\right]\right).
    \end{align}
    Now for any linear operators $a$, $b$, $c$, $d$, we have
    \begin{align}
\nonumber
        \left[ ab,cd \right]&  = \left[ ab,c \right]d + c \left[ ab,d
        \right] \\
\nonumber
        & = \left( a \left\{ b,c \right\} - \left\{ a,c \right\}b
        \right) d + c \left(a \left\{b,d  \right\} - \left\{ a,d
        \right\}b \right) \\
        \nonumber
        &= a \left\{ b,c \right\} d - \left\{ a,c \right\}bd + ca
        \left\{ b,d \right\} - c\left\{ a,d \right\} b
        \\& = a \left\{ b,c \right\}d - \left\{ a,c \right\}bd - ac
        \left\{ b,d \right\} - c\left\{ a,d \right\} b + \left\{ a,c
        \right\}\left\{ b,d \right\}.
        \label{eq:anticommutator}
    \end{align}
    It is also obvious that for arbitrary $\alpha$, $\beta\not=\gamma$:
    \begin{align}
        \label{eq:zetaX}
    \nonumber
          \left\{
        \zeta_s \wedge - , i_{\phi_\alpha X_s} \right\}
& = 0&  \left\{ \zeta_s \wedge -, i_{X_s} \right\} & = 1\\
         \left\{ \phi_\beta^* \zeta_s \wedge - , i_{\phi_\gamma X_s}
        \right\}& = 0& \left\{ \phi_\beta^* \zeta_s \wedge -, i_{\phi_\beta X_s}
        \right\}& =
        -1 \\ \nonumber
     \left\{ \phi_\alpha^* \zeta_s\wedge -, i_{X_s} \right\} & =0.
            \end{align}
        Therefore, using \eqref{eq:anticommutator} we get
            \begin{align*}
                \left[ L_\alpha, \Lambda_\alpha \right] & =
                \sum_{s=1}^n \left( - \phi_\alpha^* \zeta_s
                \wedge i_{\phi_\alpha X_s} +\zeta_s\wedge
                i_{X_s} - 1 -\left( \phi_\gamma^*\zeta_s \wedge
                i_{\phi_\gamma X_s} + \phi^*_\beta \zeta_s
                \wedge i_{\phi_{\beta} X_s} + 1 \right) \right)
                \\&= -2n +  \sum_{s=1}^n \left(\zeta_s \wedge
                i_{X_s} - \phi_\alpha^* \zeta_s \wedge
                i_{\phi_\alpha X_s} -
                \phi_\beta^* \zeta_s \wedge
                i_{\phi_\beta X_s} -
\phi_\gamma^* \zeta_s \wedge
                i_{\phi_\gamma X_s}
                \right).
            \end{align*}
            Now the sum in the last row operates on any fixed-degree form
            involving only elements in \eqref{eq:basisshort} by
            multiplying the form by its degree. Hence
            $$
            \left[ L_\alpha, \Lambda_\alpha \right]\omega = -
            H\omega
            $$
            for all $\omega \in \Omega^*_{H,000}$.
\end{proof}
For every cyclic permutation $(\alpha,\beta,\gamma)$ of $(1,2,3)$ we
define the operator
$$
K_\alpha =\sum_{s=1}^n \left( \phi_\alpha^* \zeta_s \wedge i_{X_s} +
\zeta_s \wedge i_{\phi_\alpha X_s }+ \phi_\gamma^* \zeta_s \wedge
i_{\phi_\beta X_s} - \phi^*_\beta \zeta_s \wedge i_{\phi_\gamma
X_s} \right).
$$
Let
$\rho_1$,\dots,  $\rho_k$ be a  sequence of elements in
\eqref{eq:basisshort}.  Then from \eqref{eq:zeta} and
\begin{align*}
    \phi_\alpha^* \phi_\beta^*  = -\phi_\gamma^*, \ \ \ \ \ \ \   \phi_\beta^* \phi_\alpha^*
    = \phi_\gamma^*,
\end{align*}
it follows that
$$
K_\alpha\left( \rho_1 \wedge \dots \wedge \rho_k\right)
=\sum_{j=1}^k \left( -1 \right)^{j+1}  \rho_1 \wedge \dots \wedge
\phi_\alpha^*\rho_j \wedge \dots \wedge \rho_k.
$$
\begin{proposition}
    \label{llambdak}
    For any cyclic permutation $(\alpha,\beta,\gamma)$ of $(1,2,3)$ we have on $\Omega^*_{H,000}\left( M \right)$
    \begin{align}
        \left[ L_\alpha, \Lambda_\beta \right]& = K_{\gamma}\\
        \left[ L_\alpha, \Lambda_\gamma \right] &= -  K_{\beta}.
        \label{eq:llambdak}
    \end{align}
    In particular $K_{\alpha}$ is globally defined, for each $\alpha\in\left\{1,2,3  \right\}$.
\end{proposition}
\begin{proof}
    We have
    \begin{align*}
        \left[ L_\alpha,\Lambda_{\beta}  \right] &=
        \sum_{s=1}^n \left(
        \left[ \zeta_s\wedge \phi_\alpha^*\zeta_s\wedge -,
        i_{X_s}i_{\phi_\beta X_s} \right]
        +\left[ \zeta_s\wedge \phi_\alpha^*\zeta_s\wedge -,
        i_{\phi_\gamma X_s}i_{\phi_{\alpha}X_s} \right]\right.
        \\&\phantom{=4\sum}\left.-\left[ \phi_\beta^* \zeta_s \wedge \phi_\gamma^* \zeta_s\wedge
        -, i_{X_s}i_{\phi_\beta X_s} \right]
        -\left[ \phi_\beta^* \zeta_s \wedge \phi_\gamma^* \zeta_s\wedge
        -, i_{\phi_\gamma X_s}i_{\phi_\alpha X_s }\right] \right).
    \end{align*}
    Now, by \eqref{eq:anticommutator} and \eqref{eq:zetaX}
    we get
    \begin{align*}
        \left[ L_\alpha, \Lambda_\beta \right] &=  \sum_{s=1}^n
        \left(
        -\phi_\alpha^*\zeta_s\wedge  i_{\phi_\beta X_s }
        + \zeta_s \wedge i_{\phi_\gamma X_s}
        - i_{X_s}\left( \phi_\gamma^* \zeta_s \wedge - \right)
        + \phi^*_\beta \zeta_s \wedge i_{\phi_\alpha X_s}
        \right)
        \\&=  \sum_{s=1}^n \left(
        \zeta_s \wedge i_{\phi_\gamma X_s} + \phi^*_\gamma \zeta_s
        \wedge i_{X_s} +  \phi_\beta^* \zeta_s \wedge i_{\phi_\alpha X_s}
        - \phi_\alpha^* \zeta_s \wedge i_{\phi_\beta X_s}
        \right) \\ & =  K_{\gamma}.
    \end{align*}
    Equation \eqref{eq:llambdak} is proved as follows. We have
\begin{align*}
        \left[ L_\alpha,\Lambda_{\gamma}  \right] &=
        \sum_{s=1}^n \left(
        \left[ \zeta_s\wedge \phi_\alpha^*\zeta_s\wedge -,
        i_{X_s}i_{\phi_\gamma X_s} \right]
        +\left[ \zeta_s\wedge \phi_\alpha^*\zeta_s\wedge -,
        i_{\phi_\alpha X_s}i_{\phi_{\beta}X_s} \right]\right.
        \\&\phantom{=\sum}\left.-\left[ \phi_\beta^* \zeta_s \wedge \phi_\gamma^* \zeta_s\wedge
        -, i_{X_s}i_{\phi_\gamma X_s} \right]
        -\left[ \phi_\beta^* \zeta_s \wedge \phi_\gamma^* \zeta_s\wedge
        -, i_{\phi_\alpha X_s}i_{\phi_\beta X_s }\right] \right).
    \end{align*}
    Again by \eqref{eq:anticommutator} we get
    \begin{align*}
        \left[ L_\alpha, \Lambda_\gamma \right] &=  \sum_{s=1}^n \left(
        -\phi_\alpha^* \zeta_s \wedge i_{\phi_\gamma X_s}
        - \zeta_s \wedge i_{\phi_\beta X_s}
        - \phi_\beta^* \zeta_s \wedge i_{X_s}
        - i_{\phi_\alpha X_s} \left( \phi_\gamma^* \zeta_s \wedge - \right)
        \right)\\
        &=-  \sum_{s=1}^n \left(
        \zeta_s \wedge i_{\phi_\beta X_s} + \phi_\beta^* \zeta_s \wedge
        i_{X_s} + \phi^*_\alpha \zeta_s \wedge i_{\phi_\gamma X_s} -
        \phi_\gamma^* \zeta_s \wedge i_{\phi_\alpha X_s}\right)\\
        &= -  K_\beta.
    \end{align*}
\end{proof}
\begin{theorem}
    \label{actionofso5}
    The linear span $\mathfrak{g}$ of the  operators $ \left\{\, L_\alpha, \Lambda_\alpha, K_\alpha, H
    \,\middle|\, \alpha = 1,2,3 \right\}$ on $\Omega^*_{H,000}\left( M
    \right)$ is a Lie
    algebra.\end{theorem}
\begin{proof}
    We have to check that $\mathfrak{g}$ is closed under taking commutators.
    Clearly it is enough to check that the commutator of any two operators
    from the set
$ \left\{\, L_\alpha, \Lambda_\alpha, K_\alpha, H
    \,\middle|\, \alpha = 1,2,3 \right\}$   lies in $\mathfrak{g}$.
It is obvious that $\left[ L_\alpha, L_\beta \right]=0$ and $\left[
\Lambda_\alpha, \Lambda_\beta \right] =0$ for any pair $1\le \alpha,
\beta\le 3$. Since $K_\alpha$ does not change the degree of forms,
$L_\alpha$ raises the degree
 by $2$ and $\Lambda_\alpha$ decreases the degree  by
$2$, we get
\begin{align}
\left[ K_\alpha, H \right] &= 0 & \left[ L_\alpha, H \right] & = 2
L_\alpha & \left[ \Lambda_\alpha , H \right] & = - 2 \Lambda_\alpha.
    \label{eq:lh}
\end{align}
Furthermore, by Proposition~\ref{llambdah} we know that $\left[
L_\alpha, \Lambda_\alpha \right] = -H$, and by
Proposition~\ref{llambdak} that $\left[ L_\alpha, \Lambda_\beta
\right] = K_{\gamma}$ for any cyclic permutation
$(\alpha,\beta,\gamma)$ of $(1,2,3)$. Therefore it is left to check
that the commutators $\left[ K_\alpha,L_\alpha \right]$, $\left[
K_\alpha, L_\beta \right]$, $\left[ K_\alpha,\Lambda_\alpha
\right]$, $\left[ K_\alpha,L_\beta \right]$ and $\left[
K_\alpha,K_\beta \right]$ for all pairs $1\le \alpha,\beta\le 3$ lie
in $\mathfrak{g}$.

For any cyclic permutation $(\alpha,\beta,\gamma)$ of $(1,2,3)$ we
have
\begin{align*}
    \left[ K_\alpha, L_\alpha \right]
    &\stackrel{\eqref{eq:llambdak}}{\Relbar\!\Relbar}
    \left[ \left[ L_\beta,\Lambda_\gamma \right], L_\alpha \right] =
    \left[ \left[ L_\beta, L_\alpha \right], \Lambda_\gamma \right] +
    \left[ L_\beta, \left[ \Lambda_\gamma, L_\alpha \right] \right]
    = [L_\beta, K_\beta]\\&\ \ \  =\  - [K_\beta, L_\beta].
\end{align*}
As $(\alpha,\beta,\gamma)$ is an arbitrary cyclic permutation of
$(1,2,3)$ we get also
\begin{align*}
    \left[ K_\beta,L_\beta \right] &= - \left[ K_\gamma, L_\gamma \right] &
    \left[ K_\gamma, L_\gamma \right] &= - \left[ K_\alpha, L_\alpha
    \right]
\end{align*}
and combining we obtain $\left[ K_\alpha,L_\alpha \right]= - \left[
K_\alpha, L_\alpha \right]$, which implies $[K_\alpha, L_\alpha]=0$
for all $1\le \alpha\le 3$. Similarly, we have $\left[
K_\alpha,\Lambda_\alpha \right] = 0$.

Now for any cyclical permutation $(\alpha,\beta,\gamma)$ of
$(1,2,3)$ we have
\begin{align*}
    \left[ K_\alpha, L_\beta \right] &= - \left[\left[L_\gamma , \Lambda_\beta
    \right], L_\beta \right] = - \left[ L_\gamma, \left[ \Lambda_\beta,
    L_\beta \right] \right]= - \left[ L_\gamma, H \right] = - 2 L_\gamma,\\
    \left[ K_\alpha, L_\gamma \right] &= \left[ \left[
    L_\beta,\Lambda_\gamma \right], L_\gamma \right] = \left[ L_\beta,
    \left[ \Lambda_\gamma,L_\gamma \right]
    \right] = \left[ L_\beta, H \right] = 2 L_\beta,\\
    \left[ K_\alpha, \Lambda_\beta \right] &= \left[ \left[ L_\beta,
    \Lambda_\gamma \right], \Lambda_\beta \right] = \left[ \left[
    L_\beta,\Lambda_\beta
    \right], \Lambda_\gamma \right] = \left[-H, \Lambda_\gamma  \right] =-  2 \Lambda_\gamma
    ,\\
    \left[ K_\alpha,\Lambda_\gamma \right] &= - \left[ \left[ L_\gamma,
    \Lambda_\beta
    \right], \Lambda_\gamma \right] = - \left[ \left[ L_\gamma,
    \Lambda_\gamma \right], \Lambda_\beta \right] = \left[H, \Lambda_\beta  \right] = 2\Lambda_\beta,\\
    \left[ K_\alpha,K_\beta \right] &= \left[ \left[ L_\beta,\Lambda_\gamma
    \right], K_\beta \right] = \left[ L_\beta, \left[ \Lambda_\gamma,
    K_\beta \right] \right] = \left[ L_\beta, 2\Lambda_\alpha \right] = -2 K_\gamma.
\end{align*}
\end{proof}
Now we prove that  the Lie algebra $\mathfrak{g}$ can be identified with the Lie algebra
$so\left( 4,1 \right)$. Let us recall the definition of $so\left(
4,1 \right)$. We denote by $E_1$ the matrix $$\mathrm{diag}\left(
1,1,1,1,-1 \right).$$ Then
$$
so\left( 4,1 \right) := \left\{\, A\in M_5\left( \mathbb{R} \right)
\,\middle|\, AE_1 = - E_1 A^t \right\}
$$
as a set. The Lie bracket on $so\left( 4,1 \right)$ is given by the usual
commutator of matrices. We denote by $e_{ij}$ the matrix with $1$ at
the place $\left( i,j \right)$ and zeros elsewhere. Define for $1\le
i<j\le 5$
\begin{align*}
    t_{ij}=
    \begin{cases}
        e_{i5} + e_{5i} & j = 5\\
        e_{ij} - e_{ji} & \mbox{otherwise}.
    \end{cases}
\end{align*}
Then the set $\left\{\, t_{ij} \,\middle|\, 1\le i<j\le 5 \right\}$
is a basis of $so\left( 4,1 \right)$. A direct computation shows that
\begin{align*}
    \left[ t_{ij}, t_{ik} \right] & = -t_{jk} &
    \left[ t_{ij}, t_{jk} \right] & = t_{ik} &
    \left[ t_{ik}, t_{jk} \right] & = -t_{ij} && i<j<k<5\\
    \left[ t_{ij}, t_{i5} \right] & = -t_{j5} &
    \left[ t_{ij}, t_{j5} \right] & = t_{i5} &
    \left[ t_{i5}, t_{j5} \right] & = t_{ij} &&i< j<5
\end{align*}
We will also use $t_{ji}$ to denote $-t_{ij}$ for $1\le i<j\le 4$.
Now for any cyclic permutation $(\alpha,\beta,\gamma)$ of $(1,2,3)$
we have
\begin{align*}
    \left[  t_{\alpha 5}+t_{\alpha 4}  ,
     t_{\alpha 5} - t_{\alpha 4}  \right] &=
     \left[ t_{\alpha 5},- t_{\alpha 4} \right] + \left[ t_{\alpha 4},
     t_{\alpha 5} \right] =  -2 t_{45}
     \\
     \left[ t_{\alpha 5} + t_{\alpha 4}, 2t_{45} \right] &= 2\left(
     t_{\alpha 4} + t_{\alpha 5}
     \right)\\
     \left[ t_{\alpha 5} - t_{\alpha 4}, 2t_{45} \right] & = 2\left(
     t_{\alpha 4} - t_{\alpha 5}
     \right) = -2 \left( t_{\alpha 5} - t_{\alpha 4} \right)\\
     \left[ t_{\alpha 5} + t_{\alpha 4}, t_{\beta 5} + t_{\beta 4} \right]
     &= t_{\alpha\beta} - t_{\alpha \beta} = 0 \\
     \left[ t_{\alpha 5} + t_{\alpha 4}, t_{\beta 5} - t_{\beta 4} \right]
     &= t_{\alpha\beta} + t_{\alpha \beta} = 2 t_{\alpha\beta}\\
     \left[ t_{\alpha 5} + t_{\alpha 4}, t_{\gamma 5} - t_{\gamma 4}  \right]
     &= - 2 t_{\gamma, \alpha}\\
     \left[ 2t_{\beta\gamma}, t_{\beta 5} + t_{\beta 4} \right] &= -
     2\left( t_{\gamma 5 + t_{\gamma 4}} \right)\\
     \left[ 2t_{\beta\gamma}, t_{\gamma 5} + t_{\gamma 4} \right] &=
     2\left( t_{\beta5} + t_{\beta 4} \right)\\
     \left[ 2t_{\beta\gamma}, t_{\beta 5} -t_{\beta 4} \right] &= -2\left(
     t_{\gamma 5 }- t_{\gamma 4} \right)\\
     \left[ 2t_{\beta\gamma} , t_{\gamma 5} - t_{\gamma 4} \right] &= 2
     \left( t_{\beta 5} - t_{\beta 4} \right).
\end{align*}
Therefore the assignment
\begin{align*}
    H &\mapsto 2t_{45} & L_\alpha &\mapsto t_{\alpha 5} + t_{\alpha 4} &
    \Lambda_\alpha &\mapsto t_{\alpha 5} - t_{\alpha 4} & K_\alpha &\mapsto
    2t_{\beta\gamma}
\end{align*}
induces an isomorphism of Lie algebras $so\left(4,1  \right)\to
\mathfrak{g}$.
Thus we  have proved the following result.
\begin{theorem}
    \label{so41action}
    The operators $L_\alpha$, $\Lambda_\alpha$,  $\alpha\in\left\{ 1,2,3 \right\}$, give a
    structure of $so\left( 4,1 \right)$-module on $\Omega^*_{H,000}\left( M
    \right)$.
\end{theorem}

\section{Action of $\mathbb{H}$ on $\Omega^{2k+1}_{H,000}\left( M
    \right)$ and Betti numbers of compact 3-cosymplectic manifolds}
    \label{quaternions}

Let $U\subset M$ be an open subset and
\begin{align*}
    \left\{ \zeta_1, \phi^*_1\zeta_1, \phi^*_2\zeta_1, \phi^*_3 \zeta_1, \dots, \zeta_n, \phi^*_1 \zeta_n,
    \phi^*_2 \zeta_n, \phi^*_3 \zeta_n, \eta_1, \eta_2, \eta_3 \right\}
\end{align*}
an orthonormal basis of $1$-forms on $U$. Define $\Omega^*_{000}\left( U
\right)$ as a linear span with coefficients in $C^{\infty}\left( U
\right)$ of the set
\begin{align*}
 Y :=   \left\{ \zeta_1, \phi^*_1\zeta_1, \phi^*_2\zeta_1, \phi^*_3 \zeta_1, \dots, \zeta_n, \phi^*_1 \zeta_n,
    \phi^*_2 \zeta_n, \phi^*_3 \zeta_n\right\}.
\end{align*}
Then $\Omega^*_{H,000}\left( U \right)$ is a subspace of $\Omega^*_{000}\left(
U \right)$.
Define the operator $I_\alpha$ on $\Omega^*_{000}\left( U \right)$
extending by linearity the map
\begin{align*}
    \rho_1 \wedge \dots \wedge \rho_k & \mapsto \phi^*_\alpha \rho_1 \wedge \dots
    \wedge \phi^*_\alpha \rho_k & \rho_1, \dots, \rho_k \in Y.
\end{align*}

\begin{proposition}
    The operators $I_\alpha$, $\alpha\in \left\{ 1,2,3 \right\}$, are
    well-defined on $\Omega^*_{000}\left( M \right)$.  Moreover, they
    preserve harmonic forms. In particular, we can consider $I_\alpha$ as an
    endomorphism of $\Omega^*_{H,000}\left( M \right)$.
\end{proposition}
\begin{proof}
    For $1\le s \le k$, we define the operators $K_{\alpha,s}$ on $\Omega^k_{000}\left( U
    \right)$
    extending by linearity  the map
    \begin{align*}
        \rho_1 \wedge \dots \wedge \rho_k & \mapsto \sum_{1\le j_1<\dots
        <j_s\le k} \left( -1 \right)^{j_1 + \dots + j_s +  s} \rho_1 \wedge
        \dots \wedge \phi^*_\alpha \rho_{j_1} \wedge \dots \wedge
        \phi^*_\alpha \rho_{j_s} \wedge \dots \wedge \rho_k,
    \end{align*}
    where $\rho_1,
        \dots, \rho_k \in Y$.
    We also denote the identity operator by $K_{\alpha,0}$.
    Then $K_{\alpha,1} = K_\alpha$ and $K_{\alpha,k} = \left( -1
    \right)^{\binom{k+1}{2}} I_\alpha$.
    It is easy to check  in  local coordinates that
    \begin{align*}
        K_\alpha K_{\alpha,s} = (s+1) K_{\alpha, s+1} - (k-s+1)
        K_{\alpha, s-1}.
    \end{align*}
    These formulae can be used to show that $K_{\alpha,s}$ is a polynomial
    in $K_\alpha$ with constant coefficients which do not depend on the
    used local chart.  Since $K_\alpha$ are globally defined and
    preserve harmonic forms we get that
    the operators $K_{\alpha,s}$ are globally defined and preserve harmonic
    forms for all $s$. In
    particular, $I_\alpha$ is a well-defined operator on
    $\Omega^*_{000}\left( M \right)$ and preserves harmonic forms.
\end{proof}

It is straightforward to see that the operators $I_\alpha$, $\alpha\in
\{1,2,3\}$, restricted to $\Omega^{{ odd}}_{H,000}\left( M \right)$ satisfy the same
relations as the
imaginary
units of the quaternion algebra $\mathbb{H}$.
Therefore we get
\begin{theorem}
    \label{action_quaternions}
    Let $k$ be odd. Then $\Omega^k_{H,000}\left( M \right)$ is an
    $\mathbb{H}$-module.
\end{theorem}
\begin{corollary}
    Let $k$ be odd. Then $b_k^h$ is divisible by $ 4$.
\end{corollary}
\begin{proof}
    Every finite dimensional module over $\mathbb{H}$ is a direct sum of
    regular modules. As the dimension of the regular module is $ 4$, the
    result follows.
\end{proof}
We denote by  $(d)$ the principal ideal in $\mathbb{Z}$ generated by
$d$. In other words, $(d)$ will be the set of the integers divisible
by $d$.
\begin{corollary}
    Let $M$ be a compact $3$-cosymplectic manifold.
    For any odd $k$ we have
    $$b_{k-1} + b_k \in \left( 4 \right).$$
\end{corollary}
\begin{proof}
    Using \eqref{eq:betti} we get
    for $k=1$
    \begin{align*}
    b_0 +   b_1  = b_0^h + b_1^h  + 3b_0^h = b_1^h + 4 b_0^h \in \left( 4
    \right).
    \end{align*}
   Similarly,  for $k=3$ we get
    \begin{align*}
        b_2 + b_3 = b_2^h + 3 b_1^h + 3b_0^h + b_3^h + 3b_2^h + 3b_1^h +
        b_0^h = b_3^h + 4b_2^h + 6 b_1^h + 4 b_0^h \in \left( 4
        \right).
    \end{align*}
    Finally, for odd $k\ge 5$ we have
    \begin{align*}
        b_{k-1} + b_k& = b_{k-1}^h + 3 b_{k-2}^h + 3b_{k-3}^h +
        b_{k-4}^h + b_k^h + 3 b_{k-1}^h + 3b_{k-2}^h + b_{k-3}^h\\ & =
        b_{k}^h + 4 b_{k-1}^h + 6b_{k-2}^h + 4b_{k-3}^h + b_{k-4}^h \in
        \left( 4 \right).
    \end{align*}
\end{proof}

\section{Inequalities on Betti numbers}
\label{inequalities}
In this section we give a bound from below on the Betti numbers of a
compact $3$-cosymplectic manifold. We start with the following
statement about horizontal Betti numbers, which is a generalization
of Wakakuwa's Theorem~9.1 in \cite{wakakuwa}.
\begin{proposition}
    \label{waka}
    Let $M$ be a compact $3$-cosymplectic manifold of dimension $ 4n+3$.
    Then for $0\le k\le n$
    $$
    b_{2k}^h \ge \binom{k+2}{2}.
    $$
\end{proposition}
\begin{proof}
    Recall the definition \eqref{eq:xi} of the $2$-forms $\Xi_\alpha$.
Let us fix $0\le k\le n$. We consider the set
$$
S_k := \left\{\,  \Xi_1^{k_1} \wedge \Xi_2^{k_2}\wedge \Xi_3^{k_3}
\,\middle|\, k_1 + k_2 + k_3  = k\right\}.
$$
All the elements of $S_k$ can be obtained from the constant $0$-form
$1$ on $M$ by successive applications of operators $L_\alpha$,
$\alpha\in \left\{ 1,2,3 \right\}$.  Therefore by
Remark~\ref{horizontal} we get $S_k\subset \Omega^{2k}_{H,000}(M)$.
Thus, to prove the proposition it is enough to show that $S_k$
contains $\binom{k+2}{2}$ linearly independent elements. This can be
checked locally. Let $U$ be a trivializing neighbourhood like in
Section~\ref{action}. We order the elements of the basis
\eqref{eq:basis} of $1$-forms on $U$ by
\begin{multline*}
\zeta_1< \zeta_2 <\dots <\zeta_n < \phi_1^* \zeta_1<\phi_1^* \zeta_2<\dots
<\phi_1^*
    \zeta_n \\< \phi_2^*\zeta_1 < \dots< \phi_2^* \zeta_n< \phi_3^* \zeta_1<\dots<
    \phi_3^*
        \zeta_n< \eta_1< \eta_2<\eta_3.
\end{multline*}
    Then we get an induced lexicographical  ordering on the basis of $\Omega^k\left( U \right)$.
    By using the local expression \eqref{eq:Xi} of $\Xi_\alpha$, $\alpha\in
    \left\{ 1,2,3 \right\}$,
    we see that the first basis element with respect to this ordering that
    enters in
$\Xi_1^{k_1} \wedge \Xi_2^{k_2}\wedge \Xi_3^{k_3}$ with a non-zero
coefficient is
\begin{multline*}
    \zeta_1\wedge \phi_1^*\zeta_1 \wedge \zeta_2 \wedge \phi_1^* \zeta_2
    \wedge \dots \wedge \zeta_{k_1} \wedge
    \phi_1^*\zeta_{k_1}\wedge \zeta_{k_1+1} \wedge \phi_2^*\zeta_{k_1+1}
    \wedge
    \dots \wedge \zeta_{k_1+k_2} \wedge \phi_2^* \zeta_{k_1+k_2}
    \\ \wedge
    \zeta_{k_1+k_2+1} \wedge \phi_3^*\zeta_{k_1+k_2+1} \wedge \dots \wedge
    \zeta_{k_1+k_2+k_3} \wedge \phi_3^* \zeta_{k_1+k_2+k_3}.
\end{multline*}
    Since for different triples $\left( k_1,k_2,k_3 \right)$ such that
    $k_1 + k_2 + k_3=k$ the above basis elements are different,
    we get that $S_k$ contains $\binom{k+2}{2}$ elements and they are linearly
    independent.
\end{proof}
As a consequence we get the following lower bound on the Betti numbers of a
compact $3$-cosymplectic manifold.
\begin{theorem}
    \label{inequality}
    Let $M$
 be a compact $3$-cosymplectic manifold of dimension $ 4n+3$.
    Then for $0\le k\le 2n+1$
$$
    b_{k} \ge \binom{k+2}{2}.
    $$
\end{theorem}
\begin{proof}
    For $k=0$ we have obviously $b_0 = 1 = \binom{2}{2}$.
    First we consider the case $k = 2l$, $1\le l\le n$. Then by
    \eqref{eq:betti} and Proposition~\ref{waka}
    \begin{align*}
        b_k & = b^h_{2l} + 3 b^h_{2l-1} + 3 b^h_{2l-2} + b^h_{2l-3} \ge
        \binom{l+2}{2}+3\cdot 0  +3  \binom{l-1+2}{2} + 0 \\& =
        \frac{\left( l+2 \right)\left( l+1 \right)}{2} +
        3\frac{\left( l+1 \right)l}{2}  = \frac{\left( l+1 \right)\left(
        l+2 + 3l \right)}{2} =  \frac{\left( 2l+2 \right)\left( 2l+1
        \right)}{2}\\& = \binom{k+2}{2}.
    \end{align*}
    Now, suppose that $k= 2l+1$, $0\le l\le n$. Then, again by
    \eqref{eq:betti} and Proposition~\ref{waka}
    \begin{align*}
        b_k & = b^h_{2l+1} + 3b^h_{2l} + 3b^h_{2l-1} + b^h_{2l-2} \ge
        0 + 3\binom{l+2}{2} + 3\cdot 0 + \binom{l-1+2}{2}
        \\& = 3\frac{\left( l+2 \right)\left( l+1 \right)}{2} +
        \frac{\left( l+1 \right)l}{2} = \frac{\left( l+1 \right)\left(
        3l+6 + l \right)}{2} = \frac{\left( 2l+2 \right)\left( 2l+3
        \right)}{2} \\& =\binom{2l+3}{2} =  \binom{k+2}{2}.
    \end{align*}
\end{proof}

\section{Nontrivial examples of compact 3-cosymplectic manifolds}
\label{example} The standard example of a compact 3-cosymplectic
manifold is given by the torus $\mathbb{T}^{4n+3}$ with the
following structure (cf. \cite[page 561]{martincabrera}). Let
$\left\{\theta_{1},\ldots,\theta_{4n+3}\right\}$ be a basis of
$1$-forms such that each $\theta_{i}$ is integral and closed. Let us
define a Riemannian metric $g$ on $\mathbb{T}^{4n+3}$ by
\begin{equation*}
g:=\sum_{i=1}^{4n+3}\theta_{i}\otimes\theta_{i}.
\end{equation*}
For each $\alpha\in\left\{1,2,3\right\}$  we define a tensor field
$\phi_\alpha$ of type $(1,1)$ by
\begin{align*}
\phi_{\alpha}=\sum_{i=1}^{n}& \left( E_{\alpha
n+i}\otimes\theta_{i}-E_{i}\otimes\theta_{\alpha+i}+E_{\gamma
n+i}\otimes\theta_{\beta n +i}\right.\\
&\quad\left.-E_{\beta n +i}\otimes\theta_{\gamma n+i}\right)
+E_{4n+\gamma}\otimes\theta_{4n+\beta}-E_{4n+\beta}\otimes\theta_{4n+\gamma},
\end{align*}
where $\left\{E_{1},\ldots,E_{4n+3}\right\}$ is the dual
(orthonormal) basis of
$\left\{\theta_{1},\ldots,\theta_{4n+3}\right\}$ and
$\left(\alpha,\beta,\gamma\right)$ is a cyclic permutation of
$\left\{1,2,3\right\}$. Setting, for each
$\alpha\in\left\{1,2,3\right\}$, $\xi_{\alpha}:=E_{4n+\alpha}$ and
$\eta_{\alpha}:=\theta_{4n+\alpha}$, one can easily check that the
torus $\mathbb{T}^{4n+3}$ endowed with the structure
$(\phi_{\alpha},\xi_{\alpha},\eta_{\alpha},g)$ is 3-cosymplectic.

On the other hand, the standard example of a noncompact
3-cosymplectic manifold is given by $\mathbb{R}^{4n+3}$ with the
structure described in \cite[Theorem 4.4]{cappellettidenicola}.

Both the above examples are the global product of a hyper-K\"{a}hler
manifold with a $3$-dimensional flat abelian Lie group. In fact,
locally this is always true.

\begin{proposition}\label{localdecomposition}
Any 3-cosymplectic manifold $M^{4n+3}$ is locally the Riemannian
product of a hyper-K\"{a}hler manifold $N^{4n}$ and a
$3$-dimensional flat abelian Lie group $G^{3}$.
\end{proposition}
\begin{proof}
The tangent bundle of $M^{4n+3}$ splits up as the orthogonal sum of
the vertical distribution $\mathcal{V}$ and the horizontal
distribution $\mathcal{H}$, which define Riemannian foliations with
totally geodesic leaves. Therefore, by the de Rham decomposition
theorem the manifold $M$ is locally the Riemannian product of a leaf
$N^{4n}$ of $\mathcal H$ and a leaf $G^{3}$ of $\mathcal V$. The
structure tensors $\phi_{1}$, $\phi_{2}$, $\phi_{3}$ induce an
almost hyper-complex structure $(J_{1},J_{2},J_{3})$ on $N^{4n}$.
Furthermore, for each $\alpha\in\left\{1,2,3\right\}$ and for all
$X,X'\in\Gamma(TN^{4n})=\Gamma({\mathcal H})$,
\begin{equation*}
[J_{\alpha},J_{\alpha}](X,X')=N_{\phi_\alpha}(X,X')-2d\eta_{\alpha}(X,X')\xi_{\alpha}=0,
\end{equation*}
as $M^{4n+3}$ is normal and $\eta_{\alpha}$ is closed. Consequently,
the structure is hyper-complex. Finally, the induced metric is
clearly compatible with such a hyper-complex structure, so that
$N^{4n}$ is hyper-K\"{a}hler. On the other hand, from Lie group
theory (see e.g. \cite[page 10]{tricerri}) it follows that $G^{3}$
is an abelian Lie group. Since the Reeb vector fields are parallel,
we get
\begin{equation}\label{verticalcurvature}
R(\xi_{\alpha},\xi_{\beta})\xi_{\gamma}=\nabla_{\xi_\alpha}\nabla_{\xi_\beta}\xi_{\gamma}-\nabla_{\xi_\beta}\nabla_{\xi_\alpha}\xi_\gamma-\nabla_{[\xi_\alpha,\xi_\beta]}\xi_{\gamma}=0,
\end{equation}
Therefore $G^{3}$ is flat.
\end{proof}

When $n=0$ of course we have no splitting, and  $M$ is necessarily a
3-torus in the compact case, as it is shown in the following
proposition.

\begin{proposition}
    \label{hopf}
    Suppose $M^3$ is a compact three dimensional 3-cosymplectic manifold. Then
    $M^3$ is a three dimensional torus.
\end{proposition}
\begin{proof}
First of all $M^3$ is clearly flat. Indeed, in this case the three
Reeb vector fields span all the vector fields over the ring of
smooth functions. Furthermore, they commute with each other and are
parallel. Thus, similarly to \eqref{verticalcurvature} we get
$R(\xi_{\alpha},\xi_{\beta})\xi_{\gamma}=0$ for any triple of
indices $1\leq\alpha,\beta,\gamma\leq 3$.
 The manifold $M^3$ is orientable, since
$\eta_1\wedge\eta_2\wedge \eta_3\not=0$ is a volume form on $M^3$.
Moreover $\eta_1$, $\eta_2$, $\eta_3$ are three linear independent
harmonic forms of degree $1$, so that $b_1\left( M^3 \right)\ge 3$.
The complete list of  compact orientable Euclidean three-dimensional
manifolds was obtained in Sections 2-3 of~\cite{hantzschewendt}. The
unique manifold with $b_1\ge 3$ in this list is the three
dimensional torus.
\end{proof}

Due to Proposition \ref{localdecomposition}, it is natural to ask
whether there are examples of 3-cosymplectic manifolds which are not
the global product of a hyper-K\"{a}hler manifold with an abelian
Lie group. We will give an example of a compact 3-cosymplectic
manifold in dimension seven that is not a product of a
hyper-K\"{a}hler manifold and a three-dimensional torus. Before
describing the construction, we remind the following well-known
result.
\begin{theorem}
    \label{kodaira}
    If $M^4$ is a compact four-dimensional hyper-K\"{a}hler manifold, then $M^4$ is either a K3 surface or a
    four dimensional torus.
\end{theorem}
\begin{proof}
    From \cite[Theorem 8.1]{wakakuwa} it follows that
    $b_1(M^4)$ is even. Moreover, since every hyper-K\"{a}hler manifold is
    Calabi-Yau, $M^4$ has a trivial
    canonical bundle.  Therefore, by the Kodaira classification
    (cf.~\cite[Section 6A]{kodaira1}) $M^4$ is either a K3 surface or a 4-torus.
\end{proof}

Let $(M^{4n},J_{\alpha},G)$ be a compact hyper-K\"{a}hler manifold,
where $(J_{1},J_{2},J_{3})$ is the hyper-complex structure of
$M^{4n}$ and $G$ is the compatible Riemannian metric. Let
$f:M^{4n}\longrightarrow M^{4n}$  be a hyper-K\"{a}hlerian isometry,
that is $f$ is an isometry such that
\begin{equation}\label{map1}
f_{\ast}\circ J_{\alpha}=J_{\alpha}\circ f_{\ast}
\end{equation}
for each $\alpha\in\left\{1,2,3\right\}$.  Let us define the action
$\varphi$  of $\mathbb{Z}^{3}$ on the product manifold $M^{4n}\times
\mathbb{R}^{3}$ by
\begin{equation*}
\varphi\left(\left(k_{1},k_{2},k_{3}\right),\left(x,t_{1},t_{2},t_{3}\right)\right)=\left(f^{k_1+k_2+k_3}(x),t_{1}+k_{1},t_{2}+k_{2},t_{3}+k_{3}\right).
\end{equation*}
Note that the action $\varphi$ is free and properly discontinuous,
hence the orbit space $M^{4n+3}_f:=(M^{4n}\times
\mathbb{R}^{3})/{\mathbb{Z}^{3}}$ is a smooth manifold. We define a
3-cosymplectic structure on $M^{4n+3}_f$ in the following way. Let
$\hat{\xi}_{1},\hat{\xi}_{2},\hat{\xi}_{3}$ be the vector fields on
$M^{4n}\times{\mathbb{R}}^{3}$ given by
$\hat{\xi}_\alpha:=\frac{\partial}{\partial t_{\alpha}}$, and let
$\hat{\eta}_{1},\hat{\eta}_{2},\hat{\eta}_{3}$ be the $1$-forms
defined by $\hat{\eta}_{\alpha}:=\hat{g}(\cdot,\hat{\xi}_{\alpha})$,
where
\begin{equation*}
\hat{g}=G+dt_{1}\otimes dt_{1}+dt_{2}\otimes dt_{2}+dt_{3}\otimes
dt_{3}.
\end{equation*}
Let $\hat{\phi}_{\alpha}$ be the tensor field of type $(1,1)$ on
$M^{4n}\times \mathbb{R}^{3}$ defined  as follows. Let $E$ be a vector field on
$M$. We can uniquely decompose $E$ into the sum of a vector field $X$ tangent to
$M^{4n}$ and
its vertical part  $\sum_{\beta=1}^{3}\hat{\eta}_{\beta}(E)\hat{\xi}_{\beta}$. Then
we set
\begin{equation*}
\hat{\phi}_{\alpha}E:=J_{\alpha}X+\sum_{\beta,\gamma=1}^{3}\epsilon_{\alpha\beta\gamma}\hat{\eta}_{\beta}(E)\hat{\xi}_{\gamma}.
\end{equation*}

Since $f$ is an isometry, $\hat{g}$ descends to a Riemannian metric
on the quotient manifold $M^{4n+3}_f$. Furthermore, the vector
fields $\hat{\xi}_{1}$, $\hat{\xi}_{2}$, $\hat{\xi}_{3}$, together
with their dual $1$-forms
$\hat{\eta}_{1},\hat{\eta}_{2},\hat{\eta}_{3}$, are clearly
invariant under the action $\varphi$. Finally, because of
\eqref{map1}, also the endomorphisms $\hat{\phi}_{\alpha}$ induce
three endomorphisms on the tangent spaces of $M^{4n+3}_f$. We denote
the induced structure by
$(\phi_{\alpha},\xi_{\alpha},\eta_{\alpha},g)$,
$\alpha\in\left\{1,2,3\right\}$. By a straightforward computation
one can check that
$(M^{4n+3}_f,\phi_{\alpha},\xi_{\alpha},\eta_{\alpha},g)$ is a
3-cosymplectic manifold. Moreover, $M^{4n+3}_f$ is not in general a
global product of a hyper-K\"{a}hler manifold by the torus
$\mathbb{T}^{3}$. To see this we will consider the following more
specific seven-dimensional example.

Let $\mathbb{H}$ be the algebra of quaternions. We consider
$\mathbb{H}$ as a hyper-K\"{a}hler four-dimensional manifold with a
hyper-complex structure given by left multiplication by $\qi$,
$\qj$, $\qk$. Define the action of $\mathbb{Z}^4$ on $\mathbb{H}$ by
\begin{align*}
    \mathbb{Z}^4\times \mathbb{H} &\to \mathbb{H}\\
    \left( \left( a,b,c,d \right), q \right) & \mapsto q+ a+b\qi +
    c\qj+d\qk.
\end{align*}

 By distributivity of multiplication in $\mathbb{H}$ this action
commutes with the left multiplication by $\qi$, $\qj$, and $\qk$. Furthermore, the
Euclidean metric on $\mathbb{H}$ is translation invariant. Thus the
quotient space $\left.\raisebox{0.3ex}{
$\mathbb{H}$}\middle/\!\raisebox{-0.3ex}{ $\mathbb{Z}^4$}\right. $
is diffeomorphic to $\mathbb{T}^4$ and inherits a hyper-K\"{a}hler
structure from $\mathbb{H}$.

Let $\bar{f}\colon \mathbb{H}\to \mathbb{H}$ be the map given by the
right multiplication by $\qi$. Then from associativity of
multiplication on $\mathbb{H}$ it follows that $\bar{f}$ commutes
with the hyper-complex structure maps on $\mathbb{H}$. Moreover,
from the distributivity of multiplication in $\mathbb{H}$ it follows
that $\bar{f}$ induces a hyper-K\"{a}hlerian isometry $f$ on
$\left.\raisebox{0.3ex}{ $\mathbb{H}$}\middle/\!\raisebox{-0.3ex}{
$\mathbb{Z}^4$}\right. $.

\begin{theorem}
    Let $M^{4}= \mathbb{T}^4$ and $f$  be as above. Then $M^{7}_f$ is not the
    total space of a toric bundle over a hyper-K\"ahler four dimensional
    manifold. In particular, $M^7_f$ is not the  global product of a compact hyper-K\"{a}hler
four-manifold and the torus $\mathbb{T}^{3}$.
\end{theorem}
\begin{proof}
Suppose $M^7_f$ has a regular Reeb foliation, in other words
 $M^7_f$ is the total space of a toric bundle over a hyper-K\"ahler
 manifold $K^4= \left.\raisebox{0.3ex}{ $M^7_f$}\middle/\!\!\!\raisebox{-0.3ex}{
 $\mathcal{F}_3$}\right.$, where $\mathcal{F}_3$ is the Reeb foliation on $M^7_f$.
 Since the harmonic forms $\eta_\alpha$ generate the cohomology groups
of fibres in this bundle, by the Leray-Hirsch theorem
$$
H^*\left( M^7_f \right) \cong H^*\left( \mathbb{T}^3 \right) \otimes H^*\left(
K^4 \right).
$$
By Theorem~\ref{kodaira} there are only two possibilities for $K^4$:
either $K^4\cong \left.\raisebox{0.3ex}{
$\mathbb{H}$}\middle/\!\raisebox{-0.3ex}{ $\mathbb{Z}^4$}\right.$ or
$K^4$ is a complex K3 surface. In the first case the
Hilbert-Poincar\'e series of $H^*\left( \mathbb{T}^3 \right) \otimes
H^*\left( K^4 \right)$ is $(1+t)^7 = 1 + 7t + 21t^2 +\dots$, in the
second case it equals\begin{equation*}
    (1+22t^2 + t^4)(1+t)^3 = 1 + 3t + 25t^2 + \dots
\end{equation*}
We will show in Proposition~\ref{betti} that
 $b_2( M^7_f )< 21$. This will imply a contradiction with our initial assumption.
\end{proof}
To get an estimate on  $b_2( M^7_f )$ we will define a
structure of CW-complex on $M^7_f$.

Recall the definition of CW-complex (cf. \cite[Definition~7.3.1]{maunder}).
We will modify it by replacing the balls in $\mathbb{R}^n$ by cubes $Q^k = \left\{\,
x\in \mathbb{R}^k \,\middle|\,  0\le x_i\le 1,\ i=1,\dots,k \right\}$.
\begin{definition}
    A CW-\emph{complex} is a Hausdorff space $X$, together with an indexing
    set $I_k$ for each integer $k\ge 0$ and maps $\phi^k_\alpha\colon Q^k\to
    X$, $k\ge 0$, $\alpha\in I_k$ such that the following conditions
    are satisfied:
    \begin{enumerate}
\item $X = \bigcup_{k\ge 0} \bigcup_{\alpha\in I_k}
            \phi^k_\alpha( \mathring Q^k )$;
\item $\phi^k_\alpha( \mathring Q^k )\cap \phi^l_\beta( \mathring Q^l
    ) = \emptyset$ unless $k=l$ and $\alpha= \beta$;
\item $\phi^k_\alpha|_{\mathring Q^k}$ is one-to-one;
\item Let $X^k= \bigcup_{j\le k}\bigcup_{\alpha\in I_j}
    \phi^j_\alpha(\mathring Q^j
    )$. Then $\phi^k_\alpha( \partial Q^k )\subset
    X^{k-1}$ for each $k\ge 1$ and $\alpha\in I_k$.
\item A subset $Z$ of $X$ is closed if and only if
    $\left(\phi^k_\alpha\right)^{-1}\left( Z \right)$ is closed in $Q^k$ for each
    $k\ge 0 $ and $\alpha\in I_k$.
\item For each $k\ge 0$ and $\alpha \in I_k$ the set $\phi^k_\alpha( Q^k
    )$ is contained in the union of a finite number of sets of the
    form $\phi^l_\beta( \mathring Q^l )$.
    \end{enumerate}
\end{definition}
Let $X$ be a CW-complex. Then we have the induced maps
$$
\overline{\phi^k_\alpha }\colon S^k \cong \left.\raisebox{0.3ex}{
$Q^{k}$}\middle/\!\raisebox{-0.3ex}{ $\partial Q^{k}$}\right. \to
\left.\raisebox{0.3ex}{ $X^{k}$}\middle/\!\raisebox{-0.3ex}{ $X^{k-1}$}\right. .
$$
We will denote the image of this map by $S^k_\alpha$.
By \cite[Example~7.3.15]{maunder} we get a homeomorphism of topological spaces
\begin{align*}
    \overline{\phi^k}= \bigvee_{\alpha\in I_k} \overline{\phi^k_\alpha}\colon
    \bigvee_{\alpha\in I_k} \left.\raisebox{0.3ex}{
    $Q^k$}\middle/\!\raisebox{-0.3ex}{$ \partial Q^k$}\right.  \to
    \bigvee_{\alpha\in I_k} S^k_\alpha = \left.\raisebox{0.3ex}{
$X^k$}\middle/\!\raisebox{-0.3ex}{ $X^{k-1}$}\right. ,
\end{align*}
where $\bigvee_{\alpha\in I_k}S^k_\alpha$ denote the one point union
(see e.g.
\cite{maunder}, page 205).
 We  denote by
$q_\beta$ the map from $\bigvee_{\alpha\in I_k} S^k_\alpha$ to $S^k_\beta$ that
acts as the  identity on $S^k_\beta$ and collapses all the other spheres to the basic point.

Now we explain how the homology groups of a CW-complex can be
computed. We define $C_k\left( X \right)$ to be the free abelian
group generated by $I_k$. For every pair $\alpha\in I_k$ and
$\beta\in I_{k-1}$ we define the map $d_{\alpha,\beta}$ to be the
composition
$$
S^{k-1}\cong \partial Q^k\stackrel{\phi^k_{\alpha}}{\longrightarrow} X^{k-1}
\stackrel{\pi}{\longrightarrow}
\left.\raisebox{0.3ex}{ $X^{k-1}$}\middle/\!\raisebox{-0.3ex}{ $X^{k-2}$}\right.
=\bigvee_{\gamma\in I_k}
S^{k-1}_\gamma
\stackrel{q_\beta}{\longrightarrow}
S_\beta^{k-1}.
$$
We denote by $\left[ d_{\alpha,\beta}
\right]$ the degree of the map $d_{\alpha,\beta}$. Now define the
differential $\partial\colon C_k\left( X \right)\to C_{k-1}\left( X
\right)$ by $\partial\left( \alpha \right) = \sum_{\beta\in I_{k-1}}
\left[ d_{\alpha,\beta} \right]\beta$. It is proved in Chapter 8 of
\cite{maunder} that the homology groups of the complex $\left( C_*\left( X
\right), \partial \right)$ are isomorphic to the integral homology groups of
the space $X$.

Since $\mathbb{R}$ is torsion free and $C_k\left( X \right)$ are free
$\mathbb{Z}$-modules, it follows from the universal coefficient
theorem that
\begin{align*}
    H_k\left( C_*\left( X \right)\otimes_{\mathbb{Z}} \mathbb{R} \right)
    \cong H_k\left( C_*\left( X \right) \right)\otimes_{\mathbb{Z}}
    \mathbb{R}& = H_k^{\mathbb{R}}\left( X \right), & k\ge 0.
\end{align*}
If $X$ is an $m$-dimensional compact Riemannian manifold then we have by
the Poincar\'e duality
\begin{align*}
    H_k^{\mathbb{R}}\left( X \right) & \cong H^{m-k}_{dR}\left( X
    \right) \cong \Omega^{m-k}_{H}\left( X \right), & 0\le k\le m.
\end{align*}

Define $\pi\colon \mathbb{H}\times \mathbb{R}^3\to M^7_f$ to be the composition
$$
\mathbb{H}\times \mathbb{R}^3 \stackrel{\pi_1}{\longrightarrow} \left(
\left.\raisebox{0.3ex}{ $\mathbb{H}$}\middle/\!\raisebox{-0.3ex}{
$\mathbb{Z}^4$}\right.
\right) \times \mathbb{R}^3 \stackrel{\pi_2}{\longrightarrow} M^7_f,
$$
where $\pi_1$ and $\pi_2$ are the natural projections.

Now we describe the cellular structure on $M^7_f$. For every $k\in
\{0,\ldots, 7\}$ we denote by $I_k$ the set of $k$-subsets in
$\left\{ 1,\dots,7 \right\}$. For every $S\in I_k$ and $x\in Q^k$
define $\theta_S\left( x \right)$ to be the element of
$\mathbb{R}^7\equiv \mathbb{H}\times \mathbb{R}^3$, obtained from
$x$ by order preserving placing of coordinates of $x$ into the
places $s\in S$ and putting at all other places $0$. Now we define
$\phi^k_S\colon Q^k\to M^7_f$  to be the composition
$\pi\circ\theta_S$.
\begin{proposition}
    \label{cellular}
    The maps $\left\{\, \phi^k_S \,\middle|\, S\in I_k,\ k=0,\dots,7    \right\}$ give a CW-complex structure on $M^7_f$.
\end{proposition}
\begin{proof}
    The topological space $M^7_f$ is Hausdorff since it is a manifold.
Now we show that the restriction of $\pi$ to $\left[ 0,1 \right)^7$ is a
bijection. For any $x\in \mathbb{R}$ we denote by $\left\lfloor x\right\rfloor $
the integral part of $x$ and by $\left\{ x \right\}$ the fractional part
$x - \left\lfloor x \right\rfloor$ of $x$.

Let $\left[ \left[ q \right], \vec{x} \right]\in M^7_f$. Then $\left[ \left[
q \right], \vec{x}
\right] = \left[ \left[ q\qi^{-\left(\left\lfloor x_1 \right\rfloor +\left\lfloor x_2
\right\rfloor+ \left\lfloor x_3 \right\rfloor\right)  } \right],
\{x_1\},\{x_2\},\{x_3\}
\right]$ in $M^7_f$ by definition of the action of $\mathbb{Z}^3$ on $\left.\raisebox{0.3ex}{
$\mathbb{H}$}\middle/\!\raisebox{-0.3ex}{ $\mathbb{Z}^4$}\right.\times
\mathbb{R}^3 $. Thus the restriction of $\pi_2$ to $\left.\raisebox{0.3ex}{
$\mathbb{H}$}\middle/\!\raisebox{-0.3ex}{ $\mathbb{Z}^4$}\right. \times
\left[0,1\right)^3$ is a surjection. To see that $\pi_2$ is a bijection we note
that the map
\begin{align*}
    f\colon \left.\raisebox{0.3ex}{ $\mathbb{H}$}\middle/\!\raisebox{-0.3ex}{
    $\mathbb{Z}^4$}\right. \times \mathbb{R}^3 & \to \left.\raisebox{0.3ex}{
    $\mathbb{H}$}\middle/\!\raisebox{-0.3ex}{
    $\mathbb{Z}^4$}\right. \times \mathbb{R}^3\\
    \left( [q],\vec{x} \right) & \mapsto \left( \left[ q\qi^{-\left(\left\lfloor x_1 \right\rfloor +\left\lfloor x_2
    \right\rfloor+\left\lfloor x_3 \right\rfloor\right)  } \right],
    \left\{x_1 \right\}, \left\{ x_2 \right\}, \left\{ x_3 \right\} \right)
\end{align*}
is $\mathbb{Z}^3$-invariant. Since for different points of
$\left.\raisebox{0.3ex}{ $\mathbb{H}$}\middle/\!\raisebox{-0.3ex}{
    $\mathbb{Z}^4$}\right. \times \mathbb{R}^3 $ the values of $f$ are
    obviously different we see that the restriction of $\pi_2$ to
    $\left.\raisebox{0.3ex}{ $\mathbb{H}$}\middle/\!\raisebox{-0.3ex}{
    $\mathbb{Z}^4$}\right. \times \left[ 0,1 \right)^3 $ is injective.
Similarly we can show that the restriction of $\pi_1$ to $\left[ 0,1
\right)^7$ gives a bijection between $\left[ 0,1 \right)^7$ and
$\left.\raisebox{0.3ex}{ $\mathbb{H}$}\middle/\!\raisebox{-0.3ex}{
    $\mathbb{Z}^4$}\right. \times \left[ 0,1 \right)^3 $. Thus we get that the
    restriction of $\pi$ to $\left[ 0,1 \right)^7$ gives a bijection between
    $\left[ 0,1 \right)^7$ and $M^7_f$.

    Now we check that the maps $\phi^k_S$
    satisfy the properties of  CW-structure.
\begin{enumerate}
    \item We have
        \begin{align*}
            \bigcup_{k=0}^7 \bigcup_{S\in I_k} \theta^k_S\left(
            \mathring Q^k
            \right) = \left[ 0,1 \right)^7,
        \end{align*}
        which implies that the similar union with $\phi^k_S$ in place of
        $\theta^k_S$ gives $M^7_f$.
    \item Let $S\in I_k$ and $T\in I_l$. Then the points of $\theta^k_S\left(
        \mathring Q^k \right)$ have non-integer coordinates at
        places $s\in S$ and integer coordinates in all other
        places. Similarly for the points of $\theta^l_T\left(
        \mathring Q^l \right)$. This implies that if $S\not=T$
        then there are no common points in the sets
        $\theta^k_S\left( \mathring Q^k \right)$ and
        $\theta^l_T\left( \mathring Q^l \right)$. As the
        restriction of $\pi$ to $\left[ 0,1 \right)^7$ is a
        bijection the same property holds for $\phi^k_S\left(
        \mathring Q^k \right)$ and $\phi^l_T\left( \mathring Q^l
        \right)$.
    \item As $\theta^k_S\left( \mathring Q^k \right) \subset \left[ 0,1
        \right)^7$ we see that the restriction of $\phi^k_S$ to
        $\mathring Q^k$ is one-to-one, for any $0\le k\le 7$, $S\in
        I_k$.
    \item From the considerations at the beginning of the proof we can see
        that if two points $\left( q,x \right)$, $\left( q',x'
        \right)\in \mathbb{H}\times \mathbb{R}^3$ are representatives
        of the same point in $M^7_f$ then the number of integer
        coordinates in $\left( q,x \right)$ and $\left( q',x' \right)$
        is the same. Now $X^k\subset M^7_f$ can be identified with those
        points $\left[ \left[ q \right], x \right]\in M^7_f$ such that
    $\left( q,x \right)$ has at most $k$ fractional coordinates.
Now every point $\partial Q^k $ contains at least one integral
coordinate. Therefore for $S\in I_k$,  $\theta^k_S\left( \partial
Q^k \right)$ contains at least $7-k+1 = 8-k$ integral coordinates,
or, in other words, at most $k-1$ non-integral coordinates. Thus
$\phi^k_S \left( \partial Q^k \right) = \pi\circ \theta^k_S\left(
\partial Q^k \right)$ is a subset of $X^{k-1}$.
\item  If $Z\in M^7_f$ is closed then for any $0\le k\le 7$ and $S\in I_k$ the sets $\left( \phi^k_S
    \right)^{-1}\left( Z \right)$ are obviously closed, as the maps
    $\phi^k_\alpha$ are continuous. Suppose now that for every $0\le k\le 7$
    and $S\in I_k$ the sets $\left( \phi^k_S \right)^{-1}\left( Z
    \right)$ are closed.
    As $M^7_f$ has the quotient topology under the projection $\pi$, we have
    to show that $\pi^{-1}\left( Z \right)$ is a closed subset in
    $\mathbb{H}\times \mathbb{R}^3$. Let $\left( q_n, \vec{x}_n \right)$ be
    a sequence in $\pi^{-1}\left( Z \right)$ that converges to $\left(
    q,\vec{x}
    \right)\in \mathbb{H}\times \mathbb{R}^3$. We have to show that
    $\left( q,\vec{x} \right)\in \pi^{-1}\left( Z \right)$. Let $i\in
    \left\{ 1,2,3 \right\}$. If  $x^i$ is fractional, then
    starting from some $n$ we have $\left\lfloor  x_n^i
    \right\rfloor = \left\lfloor x^i \right\rfloor$. If $x^i$ is integer
    then for infinitely many $n$ we have $ x_n^i
    <\left\lfloor x^i \right\rfloor$
    or $\left\lfloor x^i \right\rfloor \le  x_n^i$. By
    passing to an appropriate subsequence we can assume that
    for all $n$ either
    $\left\lfloor  x_n^i  \right\rfloor = \left\lfloor
    x^i \right\rfloor -1$
    or
$\left\lfloor  x_n^i  \right\rfloor = \left\lfloor
    x^i \right\rfloor$. We denote the common integer part of $
    x_n^i $ by $\widetilde{x}^i$. Define
    \begin{align*}
        q'_n & = q_n i^{-\widetilde{x}^1 - \widetilde{x}^2 -
        \widetilde{x}^3} & \left( x'_n \right)^i & = x_n^i -
        \widetilde{x}^i\\
        q' & = q i^{-\widetilde{x}^1 - \widetilde{x}^2 -
        \widetilde{x}^3} &
  \left( x' \right)^i & = x^i -
        \widetilde{x}^i.
    \end{align*}
    Then $\left( q'_n, x'_n \right)$ is a sequence of points in
    $\pi^{-1}\left( Z \right)$ that converges to $\left( q',x' \right)$.
    Moreover $\left( q',x' \right)\in \pi^{-1}\left( Z \right)$ if and only
    if $\left( q,x \right)\in \pi^{-1}\left( Z \right)$.  We also have
    $0\le \left( x'_n \right)^i <1$ and  $x^i \in \left[ 0,1 \right]$.

    Now, similarly to the considerations above, by passing to an appropriate
    subsequence we can assume that the integer parts of the coefficients of
    $q'_n$ does not depend on $n$. Denote by  $\widetilde{q}$ the quaternion with coefficients
    equal to the integer parts of $q'_n$. Define
    $q''_n = q'_n - \widetilde{q}$ and $q'' = q' - \widetilde{q}$. Then
    $\left( q''_n ,x'_n \right)\in \pi^{-1}\left( Z \right)$ converges to
    $\left( q'', x' \right)$. Moreover, $\left( q'',x' \right)\in
    \pi^{-1}\left( Z \right)$ if and only if $\left( q',x' \right)\in
    \pi^{-1}\left( Z \right)$ if and only if $\left( q, x \right)\in
    \pi^{-1}\left( Z \right)$.

    Let $S=\left\{ 1,\dots,7 \right\}$. Note that $\theta^7_{S}\colon Q^7\to
    \mathbb{R}^7$ is the identity map on
    $Q^7$. Therefore $\left( \phi^7_S \right)^{-1}\left( Z \right) = Q^7\cap
    \pi^{-1}\left( Z \right)$. Thus the intersection $Q^7\cap \pi^{-1}\left(
    Z \right)$ is closed in $Q^7$ and thus in $\mathbb{R}^7$. Since the
    sequence $\left( q''_n,x''_n \right)$ lies in $Q^7\cap \pi^{-1}\left(
    Z \right)$ we get that also its  limit $\left( q'',x''
    \right)$ is an element of $Q^7\cap \pi^{-1}\left( Z \right) \subset
    \pi^{-1}\left( Z \right)$.

    \item Obvious, as we have only finitely many cells at every dimension.
\end{enumerate}

\end{proof}

With the cellular structure on $M^7_f$ given in Proposition~\ref{cellular} we
get

\begin{proposition}
    \label{betti}
    The degree of the map $d_{ \{3,5\}, \{3\}}$ is $1$.
   Therefore
    $\partial_2\left( \left\{ 3,5 \right\} \right) \not=0$. In particular,
    \begin{align*}
        b_2\left( M^7_f \right)= \dim\left( H^{\mathbb{R}}_2\left( M^7_f
        \right) \right) \le \dim\left( \ker\left( \partial_2
        \right) \right)<21.
  \end{align*}
\end{proposition}
\begin{proof}
    Below we identify $\mathbb{R}^7$ with $\mathbb{H}\times
    \mathbb{R}^3$. Note that $X^0$ consists of one point $ \left[ \left[
    \qzero
    \right], 0, 0,0
    \right]$. Therefore $\left.\raisebox{0.3ex}{
    $X^1$}\middle/\!\raisebox{-0.3ex}{ $X^0$}\right.  = X^1$. Now we describe
    the image of $\partial Q^2$ in $X^1$ under $\phi_{ \{3,5\}}$.
     We have
    \begin{align*}
        \partial Q^2& = \left\{\, \left( 0,x \right) \,\middle|\, 0\le
        x\le 1 \right\} \cup
        \left\{\, \left( x,1 \right) \,\middle|\,  0\le x \le 1 \right\}
        \\&\phantom{=} \cup
    \left\{\, \left( 1,x \right) \,\middle|\,  0\le x \le 1 \right\} \cup
    \left\{\,  \left( x,0 \right) \,\middle|\,  0\le x\le 1 \right\}
\end{align*}
in $\mathbb{R}^2$.
Now for all $0\le x \le 1$
\begin{align*}
    \phi^2_{ \{3,5\}} \left( 0,x \right) & = \left[ \left[ \qzero \right],x, 0,
    0\right] = \phi^1_{ \left\{ 5 \right\}}\left( x \right) \in S^1_{ \{5\}} \\
    \phi^2_{ \{3,5\}} \left( x,1 \right) & = \left[ \left[ x\qj \right], 1, 0, 0
    \right] = \left[ \left[ x\qj \left( \qi \right) \right], 0, 0, 0 \right]
    = \left[ \left[ x\qk \right], 0, 0, 0 \right] = \phi^1_{ \left\{ 4
    \right\}}\left( x \right) \in S^1_{ \{4\}}\\
    \phi^2_{ \{3,5\}} \left( 1,x \right) & = \left[ \left[ \qj \right], x, 0,0
    \right] = \left[ \left[ \qzero \right], x,0,0 \right]= \phi^1_{ \left\{
    5 \right\}}\left( x \right) \in S^1_{
    \{5\}}\\
    \phi^2_{ \{3,5\}} \left( x,0 \right) &= \left[ \left[ x\qj \right], 0, 0,0
    \right] = \phi^1_{ \left\{ 3 \right\}}(x) \in S^1_{ \{3\}}.
\end{align*}
Therefore after composing $\phi_{ \{3,5\}}$ with $q_{ \{3\}}$ we get
that
for $0\le x \le 1$
\begin{align*}
    d_{ \{3,5\}\{3\}} \left( 0,x \right) &= d_{ \{3,5\}\{3\}}\left( x,1 \right) =
    d_{ \{3,5\}\{3\}} \left( 1,x  \right) = \left[ [\qzero], 0, 0,0 \right] \in
    S^1_{ \{3\}}\\
    d_{ \{3,5\}\{3\}} \left( x,0 \right) &=   \left[ \left[ x\qj \right], 0, 0,0
    \right] \in S^1_{ \{3\}}.
\end{align*}
Now it is obvious that the degree of $d_{ \{3,5\}, \{3\}}$ is one.
\end{proof}

\small
\bibliography{three-cosymplectic}
\bibliographystyle{amsplain}
\end{document}